\documentclass[final,onefignum,onetabnum]{siamart190516}

\usepackage{listings}
\usepackage[colorinlistoftodos, textwidth=2cm, shadow]{todonotes}
\usepackage{multirow}
\usepackage{bm}
\usepackage{amsmath,amssymb}
\usepackage{float}
\usepackage{array}
\usepackage{blkarray}
\usepackage{booktabs}
\usepackage[title]{appendix}
\usepackage[T1]{fontenc}
\usepackage[utf8]{inputenc}

\PassOptionsToPackage{unicode}{hyperref}
\PassOptionsToPackage{naturalnames}{hyperref}

\newsiamremark{remark}{Remark}
\newsiamremark{example}{Example}

\usetikzlibrary{calc,arrows}

\newcommand{\tikzAngleOfLine}{\tikz@AngleOfLine}
  \def\tikz@AngleOfLine(#1)(#2)#3{%
  \pgfmathanglebetweenpoints{%
    \pgfpointanchor{#1}{center}}{%
    \pgfpointanchor{#2}{center}}
  \pgfmathsetmacro{#3}{\pgfmathresult}%
  }

\newcommand{\drawTriangleA}{
  \begin{tikzpicture}[scale=1.25]
  \coordinate [label=left:${\begin{bmatrix}0\\0\end{bmatrix}}$] (A) at (-1.5cm,-1.cm);
  \coordinate [label=right:${\begin{bmatrix}A\\0\end{bmatrix}}$] (C) at (1.5cm,-1.0cm);
  \coordinate [label=right:${\begin{bmatrix}A\\B\end{bmatrix}}$] (B) at (1.5cm,1.0cm);
  \filldraw[black] (-1.5,-1) circle (2pt) ;
    \filldraw[black] (1.5,1) circle (2pt) ;
        \filldraw[black] (1.5,-1) circle (2pt) ;
  \draw[red] (A) -- node[above,rotate=34] {${(H'\!H=A'\!A + B'\!B)^{}}$} (B) --  (C);
    \draw[red] (A) -- node[below,rotate=34] {${H}$} (B) -- (C) ;
   
         \draw[blue]  (C) -- node[below] {$A$} (A);
            \draw[green]  (C) --  node[right] {$B$}(B) ;
   \draw (1.25cm,-1.0cm) rectangle (1.5cm,-0.75cm);
  \tikzAngleOfLine(A)(B){\AngleStart}
  \tikzAngleOfLine(A)(C){\AngleEnd}
  \draw[thick,double] (A)+(\AngleStart:0.3 cm) arc (\AngleStart:\AngleEnd:0.3 cm);
  \node at ($(A)+({(\AngleStart+\AngleEnd)/2}:0.5 cm)$) {$\theta$};
    \coordinate [label=center:    Triangle Trigonometry  ] (E) at (-.3cm,1.4cm);   
  \end{tikzpicture}
}

\newcommand{\drawTriangleB}{
  \begin{tikzpicture}[scale=1.25]
  \coordinate [label=left:${\begin{bmatrix}0\\0\end{bmatrix}}$] (A) at (-1.5cm,-1.cm);
  \coordinate [label=right:${\begin{bmatrix}A\\0\end{bmatrix}}$] (C) at (1.5cm,-1.0cm);
  \coordinate [label=right:${\begin{bmatrix}A\\B\end{bmatrix}}$] (B) at (1.5cm,1.0cm);
  \coordinate [label=left:${\begin{bmatrix}0\\B\end{bmatrix}}$] (D) at (-1.5cm,1.0cm);
  \coordinate [] (E) at (-.75cm,-.5cm);

  \node at  (-1.1,-1.2) {\bf cos($\theta$)};
  \node at  (-.3,-.75) {\bf sin($\theta$)};

  \node[red,rotate=20] at (-1.15,-.6){${}_1$};

  \path[draw=red,solid,line width=1.0mm,fill=red, preaction={-triangle 90,thin,red,draw,shorten >=-1mm}] (A) -- (E);
  \path[draw=blue,solid,line width=1.0mm,fill=black, preaction={-triangle 90,thin,blue,draw,shorten >=-1mm}] (A) -- (-.75,-1);
  \path[draw=green,solid,line width=1.0mm,fill=black, preaction={-triangle 90,thin,green,draw,shorten >=-1mm}] (A) -- (-1.5,-.5);

  \filldraw[black] (-1.5,-1) circle (2pt) ;
  \path[draw=green,solid,line width=0.1mm,fill=black, preaction={-triangle 90,thin,red,draw,shorten >=-1mm}] (A) -- (B);
  \path[draw=blue,solid,line width=0.1mm,fill=black, preaction={-triangle 90,thin,draw,shorten >=-1mm}] (A) -- (D);
  \path[draw=red,solid,line width=0.1mm,fill=black, preaction={-triangle 90,thin,draw,shorten >=-1mm}] (A) -- (C);
  \draw[red] (1,.6666) -- node[above left,rotate=34] {$ {\hspace*{.5in} (H'\!H=A'\!A + B'\!B)}$} (B) ;
  \draw[red] (A) -- node[below,rotate=34] {\hspace*{.3in} ${H}$} (B) ;
 \draw[dash dot]  (B) --  (C) ;
                                                   
  \draw[blue]  (C) -- node[below] {\hspace*{.5in} $A$} (A);
  \draw[green] (A) -- node[left] {$B$} (D);
  \draw [dash dot] (D) --  (B);
  \draw [dash dot]  (-1.5,-.5) --  (-.75,-.5);
  \draw [blue,dash dot]  (-.75,-1) --  (-.75,-.5);

  \tikzAngleOfLine(A)(B){\AngleStart}
  \tikzAngleOfLine(A)(C){\AngleEnd}
  \draw[thick,double] (A)+(\AngleStart:0.3 cm) arc (\AngleStart:\AngleEnd:0.3 cm);
  \node at ($(A)+({(\AngleStart+\AngleEnd)/2}:0.5 cm)$) {$\theta$};
                         
     \coordinate [label=center:     Components  ] (E) at (0cm,1.4cm);                     
  \end{tikzpicture}
}

\newcommand{\drawTikonov}{
  \begin{tikzpicture}[scale=1.25]
  \coordinate (A) at (-1.5cm,-1.cm);
  \coordinate [label=right:$0 s$] (C) at (1.5cm,-1.0cm);
  \coordinate [label=right:$\lambda s$] (B) at (1.5cm,0.5cm);
  \coordinate [label=right:$s$] (D) at (1.5cm,2.0cm);

  \draw[red, very thick] (A) -- node[below,rotate=27] {${\sqrt{c^2 + \lambda^2s^2}}$} (B);
  \draw (C) -- (B) --  (D);
  \draw[olive, very thick] (D) -- node[above,rotate=50]{$1$}  (A);
  \draw (B) -- (C);
  \draw[blue, very thick] (C) -- node[below] {$c$} (A);
  \tikzAngleOfLine(A)(B){\AngleStart}
  \tikzAngleOfLine(A)(C){\AngleEnd}
  \draw[thick,double] (A)+(\AngleStart:0.45 cm) arc (\AngleStart:\AngleEnd:0.45 cm);
  \node at ($(A)+({(\AngleStart+\AngleEnd)/2}:0.7 cm)$) {$\theta_\lambda$};  

    \filldraw[black] (-1.5,-1) circle (2pt) ;
  \filldraw[red] (1.5,0.5) circle (2pt) ;
  \filldraw[olive] (1.5,2) circle (2pt) ;
  \filldraw[blue] (1.5,-1) circle (2pt) ;

  \end{tikzpicture}
}

\newcommand{\drawTikonovGeneral}{
  \begin{tikzpicture}[scale=1.25]
  \coordinate (A) at (-1.5cm,-1.cm);
  \coordinate [label=below:$0 L$] (C) at (1.5cm,-1.0cm);
  \coordinate [label=right:$\lambda L$] (B) at (1.5cm,0.5cm);
  \coordinate [label=right:$L$] (D) at (1.5cm,2.0cm);
  \coordinate [label=right:$C_\lambda U^Tb$] (E) at (2.1cm,0.82cm) ; 
  \coordinate [label=right:$b$] (F) at (3.5cm,-1.0cm) ; 

  \draw[red, very thick]  (A) -- node[above,rotate=27] {${H_\lambda}$} (B);
 \draw (C) -- (B) -- (D) ;
  \draw[olive, very thick] (D) -- (A) ;
   \draw[dashed] (B) -- node[below,rotate=26] {} (E) -- node[below] {} (F) -- node[below]{} (C);
  \draw (B) -- node[right] {} (C) -- node[below] {$H_0=C_\lambda H_\lambda$} (A);  
  \draw[blue,very thick] (A)--(C);
  \tikzAngleOfLine(A)(B){\AngleStart}
  \tikzAngleOfLine(A)(C){\AngleEnd}
  \draw[thick,double] (A)+(\AngleStart:0.45 cm) arc (\AngleStart:\AngleEnd:0.45 cm);
  \node at ($(A)+({(\AngleStart+\AngleEnd)/2}:0.7 cm)$) {$\theta_\lambda$};  

    \filldraw[black] (-1.5,-1) circle (2pt) ;
  \filldraw[red] (1.5,0.5) circle (2pt) ;
  \filldraw[olive] (1.5,2) circle (2pt) ;
  \filldraw[blue] (1.5,-1) circle (2pt) ;
  \filldraw[black] (2.1,0.82) circle (2pt) ;
  \filldraw[black] (3.5,-1) circle (2pt) ;

  \end{tikzpicture}
}

\newcommand{\drawGraph}{
\begin{tikzpicture}[scale=1.5]

\coordinate [label=left:$v_i$] (A) at (-3, 3) ;
\coordinate [label=right:$u_i$] (C) at (0, 0) ;
\coordinate [label=left:$0$] (B) at (-3, 0);
\coordinate [label=right: \raisebox{.2in}{\mbox{$  \! \! [c_iu_i;s_iv_i]$}}] (D) at (-0.8786, 2.1213);
\coordinate [label=below:$c_iu_i$] (E) at (-0.8786, 0);
\coordinate [label=left:$s_iv_i$] (F) at (-3, 2.1213);

\draw[red, very thick, dotted] (C) arc (0:90:3cm);
\draw[blue, very thick] (B) --node[below]{{$X$}} (C); 
\draw[green, very thick] (B) --node[left]{{$Y$}} (A);
\draw[red, very thick] (B) --(D);

\draw[dotted, thick] (D) -- (E) ;
\draw[dotted, thick] (D) -- (F) ;

\tikzAngleOfLine(B)(C){\AngleStart}
\tikzAngleOfLine(B)(D){\AngleEnd}

\node at ($(B)+({(\AngleStart+\AngleEnd)/2}:.3cm)$) {\raisebox{.05in}{\footnotesize $\ \ \ \theta_i$}};  

\end{tikzpicture}
}

\newcommand{\drawAnglePlot}{
  \begin{tikzpicture}[scale=1.5]

\coordinate [label=above:all $B$] (A) at (-3, 3) ;
\coordinate [label=right:\  $\! \! \! \! \!$ all $A$] (C) at (0, 0) ;
\coordinate [label=right:\ $\! \!  \! \!$ equally {$A,B$}] (D) at (-0.8786, 2.1213);
\coordinate [label=right:\  $\! \!$mostly $A$] (E) at (-0.22836, 1.14805) ;
\coordinate [label=right:\  $\! \! \! \! \!$  \mbox{\raisebox{,1in}{mostly $B$}}] (F) at (-1.8519, 2.7716) ;

\draw[red, very thick, dotted] (C) arc (0:90:3cm) node at (20:1){};;

\draw[blue,very thick] (B) --node[below]{{$X$}} (C); 
\draw[green,very thick] (B) --node[left]{{$Y$}} (A);
\draw[red,very thick] (B) -- (D);
\draw[red,very thick] (B) -- (E);
\draw[red,very thick] (B) -- (F);

\tikzAngleOfLine(B)(C){\AngleStart}
\tikzAngleOfLine(B)(E){\AngleEnd}

\node at ($(B)+({(\AngleStart+\AngleEnd)/2}:.3cm)$) {\raisebox{.05in}{\footnotesize $\ \ \ \  \ \ \ \theta_i$}};

  \filldraw[blue] (A) circle (1pt) ;
  \draw[black, ultra thick] (A) circle (1pt) ;
  \filldraw[blue] (C) circle (1pt) ;
  \draw[black, ultra thick] (C) circle (1pt) ;
  \filldraw[blue] (D) circle (1pt) ;
  \draw[black, ultra thick] (D) circle (1pt) ;
  \filldraw[blue] (E) circle (1pt) ;
  \draw[black, ultra thick] (E) circle (1pt) ;
  \filldraw[blue] (F) circle (1pt) ;
  \draw[black, ultra thick] (F) circle (1pt) ;

\end{tikzpicture}
}

\newcommand{\R}{\mathbb{R}}

\newcommand{\G}{\begin{bmatrix}UC\\VS\end{bmatrix}}

\title{The GSVD:  Where are the ellipses?,\\ Matrix Trigonometry, and more}

\author{Alan Edelman\thanks{Department of Mathematics, MIT, Cambridge, MA 
  (\email{edelman@math.mit.edu}).}
\and Yuyang Wang\thanks{AWS AI Labs, East Palo Alto, CA 
  (\email{yuyawang@amazon.com}). Work done prior joined Amazon.}}

\usepackage{framed}

\definecolor{shadecolor}{gray}{.92}
\definecolor{incolor}{rgb}{0,0,.7}
\definecolor{outcolor}{rgb}{.65,0,0}
\definecolor{syntaxcolor}{rgb}{.65,0,0}
\usepackage{verbatim}

\newcounter{jcounter}

\newenvironment{jinput}[1][]
{\small\noindent\begin{minipage}[t]{0in}\vskip2ex\hspace*{\fill}
\end{minipage}
\begin{minipage}[t]{.9\textwidth}\vskip-3ex   \begin{shaded}    } 
{\end{shaded}\end{minipage}\vskip 2ex\par}

\begin{document}
\maketitle

\begin{abstract}

This paper provides an advanced mathematical theory of the Generalized Singular Value Decomposition (GSVD)  and its applications.
We explore the geometry of the GSVD providing a long sought for  picture which includes a horizontal and a vertical multiaxis.
We  further propose that the GSVD provides natural coordinates for the Grassmann manifold.
This paper proves a theorem showing how the finite generalized singular values do or do not relate to the singular values of $AB^\dagger$.
 
We then turn to  applications, arguing that this geometrical theory is natural for understanding existing applications and recognizing opportunities
for new applications. In particular the generalized singular vectors play a direct and as natural a mathematical role for certain applications  as
the singular vectors do for the SVD.  In the same way that experts on the SVD often prefer not to cast SVD problems as eigenproblems, we propose that the GSVD, often cast as a generalized eigenproblem, is perhaps best cast in its natural setting.
 
We illustrate this theoretical approach and the natural multiaxes (with labels from technical domains)  in the context of applications where the GSVD arises:  Tikhonov regularization (unregularized vs regularized), Genome Reconstruction (humans vs yeast), Signal Processing (signal vs noise), and statistical analysis such as Analysis of variance (ANOVA) and discriminant analysis (between clusters vs within clusters.)  With the aid of our ellipse figure, we encourage the labelling of the natural multiaxes in any GSVD problem.
\end{abstract}



\pagestyle{myheadings}
\thispagestyle{plain}

\section{Introduction}
\subsection{Prelude}
If $a\in\R^{m_1}$ and $b\in\R^{m_2}$ are  two vectors,
then the block vector  equation in $\R^{m_1+m_2}$:
\[ \begin{bmatrix}
a \\ b
\end{bmatrix}
=\begin{bmatrix}
a \\ 
0
\end{bmatrix}
+
\begin{bmatrix}
0 \\ b
\end{bmatrix}
\]
may be thought of geometrically as a hypotenuse vector decomposed as the sum of two legs of a right triangle. 
If $h=\sqrt{\|a\|^2+\|b\|^2}\ne 0$
is the length of this hypotenuse and $u=a/\|a\|,v=b/\|b\|$ are the unit direction vectors for $a,b$ then we can write
\[ \begin{bmatrix}
a \\ b
\end{bmatrix}
=\begin{bmatrix}
uc \\ 
vs
\end{bmatrix}
h,
\]
where $c$ and $s$ are the cosine and sine of the corresponding angles, namely $c=\|a\|/h$ and $s=\|b\|/h$. This is ordinary planar trigonometry of a right triangle.

For notational convenience, we will sometimes use a semicolon (``;'') to denote the stacking (or vertical concatenation) of vectors and matrices, so that
\[[a;b]=[a;0]+[0;b].\]
We note that 
$[uc;vs]
$
is a unit vector in the direction 
$[a;b].
$
The cotangent $\sigma=c/s$ is a slope which provides a measure of whether the vector is primarily in the ``$a$'' (or top) direction, or the ``$b$,'' or a mix depending on whether $\sigma$ is large, small, or in between.

The GSVD extends the above ideas to matrices.

\subsection{The GSVD}
This paper provides a new approach and understanding
of the generalized SVD 
(GSVD)~\cite{paige1981towards,van1976generalizing,chu1995rank} of two matrices $A\in\R^{m_1, n}, B\in\R^{m_2, n}$.
Generalizing the introductory paragraphs, the GSVD may be understood in the context of a generalized Pythagorean theorem with
\[
\begin{bmatrix}
A\\
B
\end{bmatrix}
= 
\begin{bmatrix}
A\\
0
\end{bmatrix} + 
\begin{bmatrix}
0\\
B
\end{bmatrix}.
\]

 We take as our definition of a GSVD,
  a decomposition of $[A;B]$ with the form
\begin{equation}
\label{eqn:GSVD}
\left[
\begin{array}{c}
A \\ 
B
\end{array}\right] = 
\left[\begin{array}{c}
UC \\ 
VS
\end{array}\right] H,
\end{equation}
where $U,V$ are square orthogonal in $\R^{m_1 , m_1}$.
$\R^{m_2 , m_2}$; $C,S$ are 1-diagonal (see Figure  \ref{fig:CS} ) such that $C'C+S'S=I_r$, and $H$ has full row rank $r$ where $r$ denotes rank($[A;B]$).  The remaining dimensions are implied, namely $C,S$ are in $\R^{m_1 , r}$,
$\R^{m_2 , r},$ and $H$ is in $\R^{r , n}$.

The SVD is so widely used that applications need not be listed.  Historically this was not always the case.  Fields such as biology, economics, and computer science could be observed learning about the SVD one-by-one
with great impact. Perhaps a kind of folklore notion is that the SVD applies any time an array  $A$ needs to be quickly compressed to the main information out, or whenever $AA'$ was lurking. We would love to foster a world where the GSVD finds applications one-by-one in many fields.  Perhaps the new folklore is that the GSVD applies when two arrays with a common dimension need to be quickly compressed or whenever two matrices $AA'$ and $BB'$ are lurking.  Of course both the SVD and GSVD underly more.

Some selected applications of the GSVD include oriented energy analysis \cite{callaerts1989signal,callaerts1990comparison,chu2003qr,de1988mathematical,de1988oriented,vandewalle2003generalized},
(here the GSVD is sometimes called by the more descriptive name QSVD for ``quotient'' SVD),  
Tikhonov regularization \cite{hansen1989regularization,dykes2014simplified}, Linear Discriminant Analysis \cite{park2005nonlinear,howland2003structure}, and more recently in microarray analysis \cite{alter2003generalized}.
A review from 1992 and discussion of algorithms may be found in \cite{un1992csd}.

As a point of mathematical taste, many textbooks today still treat SVDs as a byproduct of exposition on eigenvalues.  This is unfortunate, as most of the time considerations of $AA'$ or $A'\! A$ create unnecessary mathematical baggage best abandoned.  The SVD is mature enough to live its own life separate from the symmetric eigenvalue problem.  Taking this notion one step further, the GSVD deserves to live separately from generalized eigenvalue problems or the SVD.  When a GSVD lurks, it is recommended to abandon old fashioned language  and see the true GSVD construction in full
mature light.  We take this approach in a number of examples in this paper.

\subsection{A ``GH'' decomposition}
\label{GH}

To clarify and streamline our view of the roles of the pieces of the GSVD, we propose that the GSVD be considered a GH decomposition:
\[\left[
\begin{array}{c}
A \\ 
B
\end{array}\right] 
= GH, \]
where $G=[UC;VS]$ (for Grassmann or geometric) denotes the information in the $r$-dimensional hyperplane representing the column space of $[A;B]$.
Specifically the columns of $G$ are a natural orthonormal basis for that hyperplane in $\R^{m_1+m_2}$, and the columns of $H$ are the coordinates of the columns
of $[A;B]$ in that basis.
Of course the $QR$ decomposition of $[A;B]$ has exactly the same properties, with one important difference: the $Q$ is not uniquely defined by the hyperplane,
while in the GSVD, the choice  is more or less canonical.

We further feel that the factorization into the two matrices $G$ and $H$ emphasizes the outer product rank $r$  form:
\begin{equation}
\label{gheq}
\left[
\begin{array}{c}
A \\ 
B
\end{array}\right] 
= 
\sum_{i=1}^r \left[ i^{\mbox{th}} \mbox{ column of }\G\right]\left[i^{\mbox{th}} \mbox{ row of } H\right],
\end{equation}
which can be readily missed in the long form.

In analogy with the SVD or Non-negative Matrix Factorization (NMF)~\cite{lee_learning_1999},  one might consider a simultaneous rank reducing method where only the $k$ rows of $H$ with largest norm are kept.

In particular if we multiply $[A;B]$ on the right by   $H^\dagger I_{r,k}^{} I_{r,k}' H$, where $I_{r,k}$ is the first $k$ columns of the $r \times r$ identity,
we obtain a rank reduced $[A;B]$:
\begin{equation*}
\begin{split}
\left[
\begin{array}{c}
A \\ 
B
\end{array}\right]
& \approx    \left( \G I_{r,k}^{} \right)  \left(  I_{r,k}' H  \right)\\ 
&= \sum_{i=1}^k \left[ i^{\mbox{th}} \mbox{ column of }\G\right]\left[i^{\mbox{th}} \mbox{ row of } H\right].
\end{split}
\end{equation*}

 We remark that  $H^\dagger I_{r,k}^{} I_{r,k}' H$ is an oblique projector when $H$ is square non-singular, and an orthogonal projector when $H$ is orthogonal.

\subsection{\texorpdfstring{More details about $U,V,C,S,H$}{}}
The matrices $U,V,C,S,H$ deserve more detailed discussion, as may be found in Appendix~\ref{sec:appendix}.

To help guide the reader, we offer a table of bases for the fundamental subspaces that appear in the GSVD.  It is helpful to keep in mind
that the columns of $C$ and $S$ are leftward looking towards the orthogonal $U$ and $V$ matrices in the GSVD factorization, while
the rows of $C$ and $S$ are rightward looking towards the full row rank $H$ in the GSVD factorization.\\

\noindent
\begin{tabular}{m{3.8cm}m{8.2cm}}
\toprule
Fundamental Spaces & Basis (with Link to $C,S$) \\ \midrule
Column Spaces of $A,B$: & Columns of $U,V$ corresponding to non-zero cols of $C,S$ \\
Left-Null Spaces of $A,B$: & Columns of $U,V$ corresponding to zero cols of $C,S$  \\
Row Space of $[A;B]$: & Rows of $H$ \\
Row Spaces of $A,B$: & Rows of $H$ corresponding to non-zero rows of $C,S$ \\
Null Spaces of $A,B$: & Columns of $H^\dagger$ corresponding to zero columns of $C,S$  \\ & \hspace*{1in} + common null space (if $r<n$)  \\
Gen Eigenvector Spaces: &  Columns of $H^\dagger$  (for the problem $\det(A'A-\lambda B'B)=0$)\\
Common Null Space: & Null space of $H$  (Also see \ref{orthtri} for an RQ drilldown) \\
\bottomrule
\end{tabular}\\

It is useful to point out that the common nullspace of $A$ and $B$ is killed by $H$, i.e., if $Ax=0$ and $Bx=0$ then $Hx=0$.
A vector that is in only one of the nullspaces is not killed by $H$, but $Hx$ is killed by 0 columns in $C$ or $S$ respectively.

Let  $r_a = \text{rank}(A), r_b = \text{rank}(B), r=\text{rank}[A;B]$.
Table  \ref{tab:cs} shows the structure of $C$ and $S$.  A very common case has $r=n$ in which case the sizes of $C,S$
 match that of $A,B$.

\begin{table}[htp]
\begin{center}
    \begin{tabular}{l|c|c}
    \toprule
    \textsc{Property of $C$ and $S$} & \textsc{$C$} & \textsc{$S$}\\
     
  \midrule

         total  \# columns  &  \multicolumn{2}{c}{$r$}\\   \hline

 \# zero columns in $S$    (left columns):   & \multicolumn{2}{l}{$r - r_b   = \ \ \ \  \  \ $\# $\{ c_i=1 \} = \# \{ s_i=0 \} $ } \\
 \# non-zero columns  (middle columns): & \multicolumn{2}{l}{$r_a + r_b - r = \ \ \ \ \  \#\{ 0<c_i,s_i<1 \}$ } \\
 \# zero columns in $C$ (right columns): &  \multicolumn{2}{l}{$r - r_a  = \ \ \ \ \ $\# $\{ c_i=0 \} = \# \{ s_i=1 \} $ }  \\ \midrule
   total   \# rows & $m_1$=\ \# rows $A$  & $m_2$=\ \# rows $B$ \\ \midrule
\# non-zero rows  &$r_a \le m_1 $ & $r_b \le m_2$ \\
\# zero rows &$m_1-r_a$ & $m_2-r_b$ \\
          \bottomrule
    \end{tabular}
\end{center} 
\label{tab:cs}
\caption{The $C$ and $S$ matrices are naturally simultaneously partitioned into three block columns such that
the number of columns $r = (r-r_b) + (r_a+r_b-r)+(r-r_a)$, in left to right order.  The row sizes conform to $A$ and $B$ which means
that we add rows of zeros to $C,S$ or possibly delete some of the zero cosines/sines to achieve a row count of $m_1,m_2$.
The number of non-degenerate angles (not 0 nor $\pi/2$) is the middle number $ (r_a+r_b-r)$.
}
\end{table}

\begin{figure}[H]
\begin{center}
\includegraphics[width=.7\textwidth]{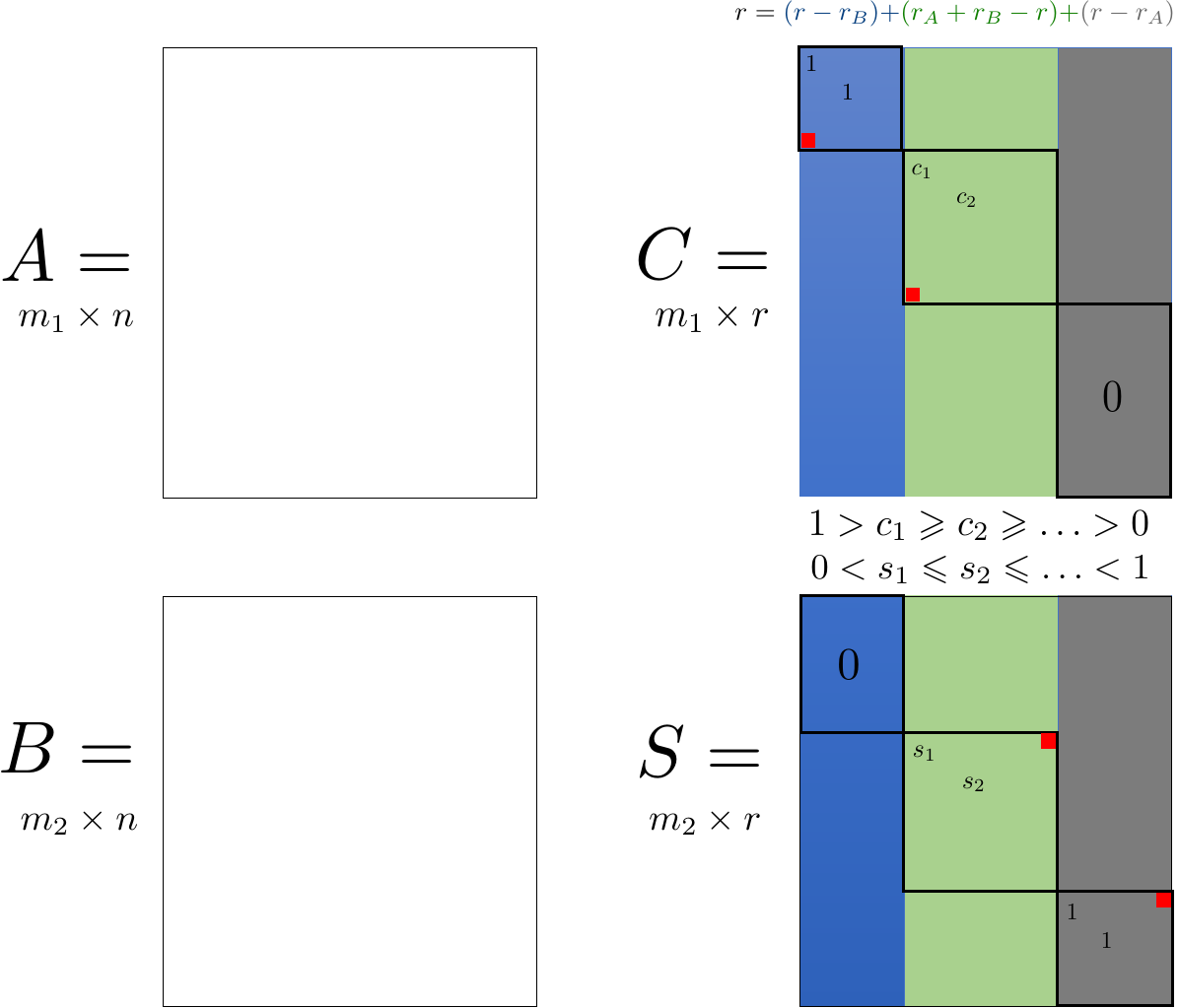}
\end{center}
\caption{Visualization of the structure of the  $C$ and $S$ matrices (whose sizes are that of $A$ and $B$). Red squares denote square blocks. We prefer the matrix diagonal orderings consistent with the cosine and sine functions on $[0,\pi/2]$,
where cosine (sine) decreases (increases) from 1 to 0  (0 to 1, respectively).}
\label{fig:CS}
\end{figure}

\subsection{Summary}
This paper contains a number of insights and results about the GSVD:

\begin{itemize}
\item We present an ellipse picture of the GSVD, which requires four dimensions to get a good feel for the general case (Section~\ref{sec:ellipse}).
\item The GSVD generalizes planar trigonometry to matrix trigonometry (Section~\ref{sec:main}).
\item We consider $[UC;VS]$ as natural coordinates for $r$ dimensional hyperplanes (the Grassmann manifold) in $\R^m$ given that $m=m_1+m_2$. We use the Grassmann manifold coordinates to clarify the link between the CS decomposition and the GSVD (other authors have observed vaguely that they are closely related). We view the $H$ matrix as the change of coordinates from canonical coordinates $[UC;VS]$ to the specifics of $[A;B]$ (Section~\ref{sec:cs}). 
\item We discuss the link between the GSVD and the principal angles between subspaces (Section~\ref{sec:pa}), and related ``energy portraits'' (Section~\ref{sec:energy}).
\item We prove a theorem relating GSVD$(A,B)$ and SVD$(AB^\dagger)$. They are not generally identical (Section~\ref{sec:thm}).
\item We revisit applications in the geometric context, and interpret the GSVD as a multi-dimensional slope and connect applications (Section~\ref{sec:applications}).
\end{itemize}

\paragraph{Notation}  For $i=1,\dots,r$,  let $u_i$ denote the normalized $i$-th column of $UC$ if $c_i \ne 0$, or else define $u_i=0$.
Similarly, let $v_i$ denote the normalized $i$-th column of $VS$ if $s_i \ne 0$, or else define $v_i=0$.
This notation conveniently avoids issues of different sizes and conventions.  For example, $U$ or $V$ may have fewer than $r$ columns.
Details of the placement of the $c_i$ and $s_i$ appear in Figure~\ref{fig:CS}. Suffice it to say for now that $u_i$ is the $i$-th column of the $U$ matrix when $c_i>0$,
and $v_i$ may be found in the $k$-th column of the $V$ matrix when $S_{ki}=s_i>0$.  The indirection in $V$ is admittedly unfortunate, but in all cases, the  non-zero $v_i$ by convention
are left to right contiguous columns of $V$ that may either start from the left, or end at the right, but in many situations $v_i$ is not in the $i$-th column. We use $A^\dagger$ to denote the pseudo-inverse of $A.$ The ``slash'' and ``backslash'' are defined as $A\backslash B := A^\dagger B,$ and $A/B := AB^\dagger.$ We also overload the notation $\text{GSVD}(A, B)$ to denote the generalized singular values of $(A, B),$ while $\text{SVD}(A)$ means the singular values of $A$. 

\section{Where are The Ellipses?}
\label{sec:ellipse}
The SVD ellipse picture for a matrix $A$ (Figure~\ref{fig:SVD}) is a very familiar visual for the action of $A$ on the unit ball. We are not aware of any ellipse pictures in the literature nor even a notion that a natural ellipse picture exists for the GSVD or even the CSD (CS Decomposition)~\cite{golub2012matrix}. We believe that the lack of a geometric view of the GSVD is part of the reason that the GSVD is not as widely understood or as widely used as it should be.  

\begin{figure}[H]
\begin{center}
\includegraphics[width=.7\textwidth]{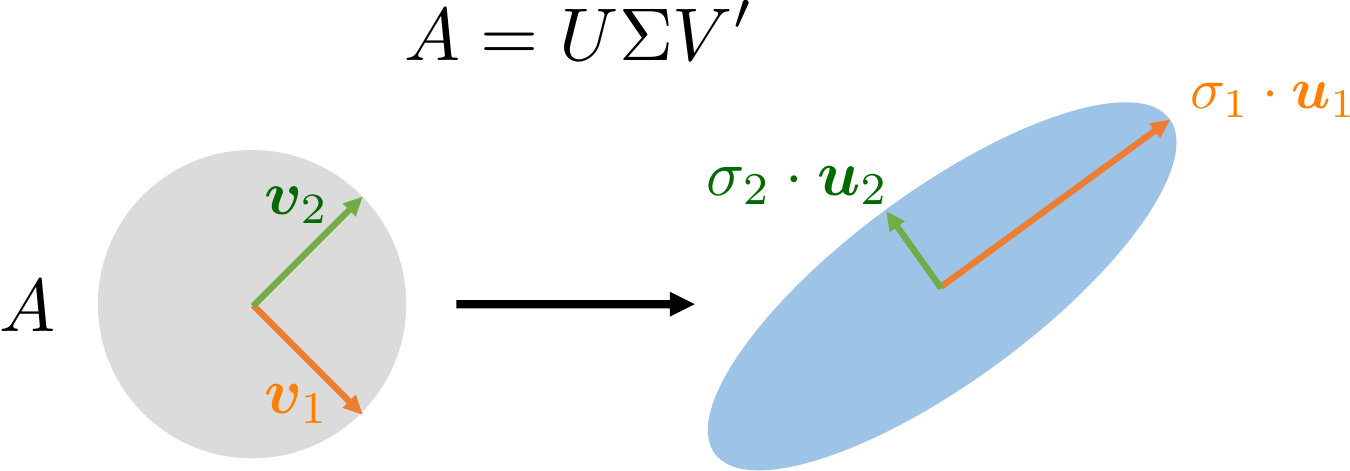}
\end{center}
\caption{Familiar SVD visual showing singular vectors and singular values of a matrix $A$ through the action of $A$ on the unit ball.  }
\label{fig:SVD}
\end{figure}

Regarding an ellipse picture, one might blame some sort of human inability to perceive higher dimensions as a complication, but we show that this is not really the case in Figure~\ref{fig:elip}.

\definecolor{mygreen}{RGB}{1, 100, 1}

\begin{figure}[htp]
\begin{center}
\includegraphics[width=\textwidth]{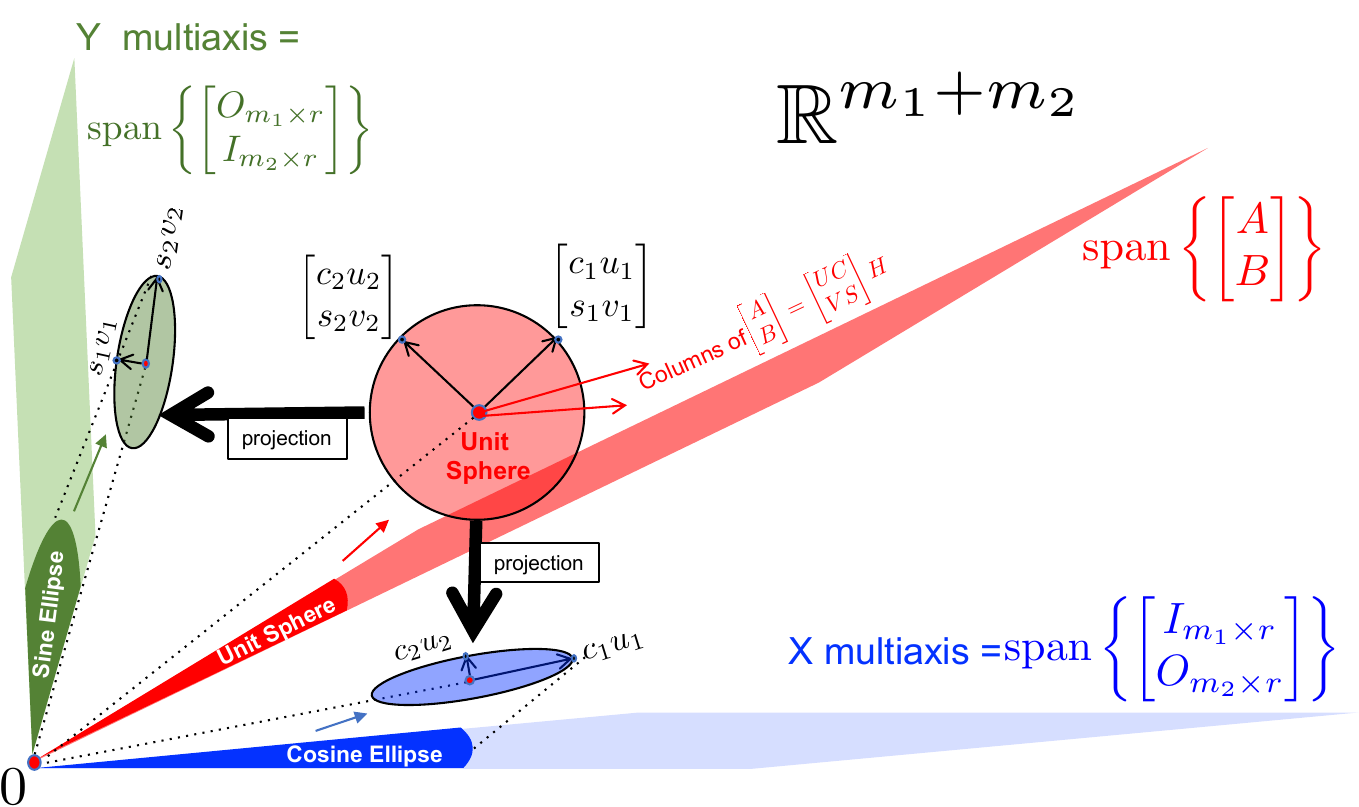}
\end{center}
\caption{Ellipse picture for the GSVD (illustrated generically in four dimensions with $m_1=m_2=r=r_a=r_b=2$): containing a red plane (span of $[A;B]$), and the $X$ (blue) and $Y$ (green) multiaxes.
Centered at the origin is a unit sphere (light red) and two ellipses (blue and green) shown in exploded view format.
The ellipses, which may be named the cosine and sine ellipses are  ``horizontal'' and ``vertical'' shadows of the unit sphere. \newline \newline
Color Coding (consistent for all figures in this paper): \newline 
\color{red} shade of RED=Span($[A;B]$), 
\color{blue} shade of BLUE=$X$ Multiaxis, 
\color{mygreen} shade of GREEN=$Y$ Multiaxis.
}
\label{fig:elip}
\end{figure}

The gap in understanding is underscored by the curiosity expressed online, but without answer, on such sites
as MATLAB Central~\cite{dyas2000}   (reproduced here\footnote{The authors contacted Mr. Dyas  on December 26, 2019 to inform him of the solution of his twenty year query.}) and  a similar request on the question-and-answer site Quora~\cite{quora2014}  (not reproduced here).

\begin{quotation}
\begin{verbatim}


Subject: Generalized SVD geometry?
From: Bob Dyas
Date: 29 Feb, 2000 15:31:31
\end{verbatim}
\noindent
\verb+Message: 1 of 1+ {\color{red} $\longleftarrow$ indicates no answer in 20 years!}

\noindent
\begin{verbatim}

Is there a geometric interpretation of the generalized singular
value decomposition? I'm looking for something comparable to 
the geometry associated with the standard SVD. I understand how 
U, V and the singular values of the SVD relate to the geometry
of the input matrix but I don't have an intuitive feel for how 
U, V, X and the generalized singular values relate to the 
geometry of the two input matrices of the GSVD.

Any help would be appreciated.
-- 
Bob Dyas 
\end{verbatim}
\end{quotation}

\subsection{Understanding the Ellipse Picture for the GSVD}

Figure 2, portrayed in four dimensional space, generically serves to illustrate the GSVD in any dimensions.

Given $A \in \R^{m_1, n},B \in \R^{m_2, n}$, we consider the unit sphere  (shown in exploded form in Figure 2
as a red circle) in the span of $[A;B]$ (shown as a red plane).
In blue and green  we have the ellipses that show the ``downward'' and ``leftward'' projections of these ellipses onto
the multiaxes $X$ and $Y$ defined as those vectors whose first $m_1$ or last $m_2$ coordinates may not vanish.
(For example if $m_1=m_2$ in $\R^4$, then the $X$ multiaxis consists of vectors of the form $(x_1,x_2,0,0)$ and
the $Y$ multiaxis consists of vectors of the form $(0,0,x_3,x_4)$.

The $u_i,v_i$ are semi-axes of these ellipses, with lengths $c_i,s_i$.  The vector $[u_i c_i ; v_i s_i]$ is on the (red)
unit sphere  in the span of $[A;B]$.

Since we have the equality $[A;B]x = [UC;VS]Hx$, we see that $H$ is the change of coordinates from the columns of $[A;B]$
to the orthonormal columns of $[UC;VS]$, and $H^\dagger$ goes the other way.

\subsection{An in depth look at  small dimensional special cases}

\subsubsection{\texorpdfstring{A red line in $\R^2$,  $X$=the $x$-axis, $Y$=the $y$-axis}{}}

$(m_1=m_2=n=r=1)$

Below we show the possibilities for $[C;S]$ for a line in $\R^2$  (drawn in red as the span of $[a,b]$ where $a$ and $b$ are $\in \R^{1}$)  which may
be horizontal $a\ne 0, b=0$, general position $a\ne 0, b\ne 0$, or vertical
$a=0, b\ne 0.$  In any event the $c$ and $s$ are the cosine and sine of the angle with the horizontal.

\begin{center}
\includegraphics[width=.9\textwidth]{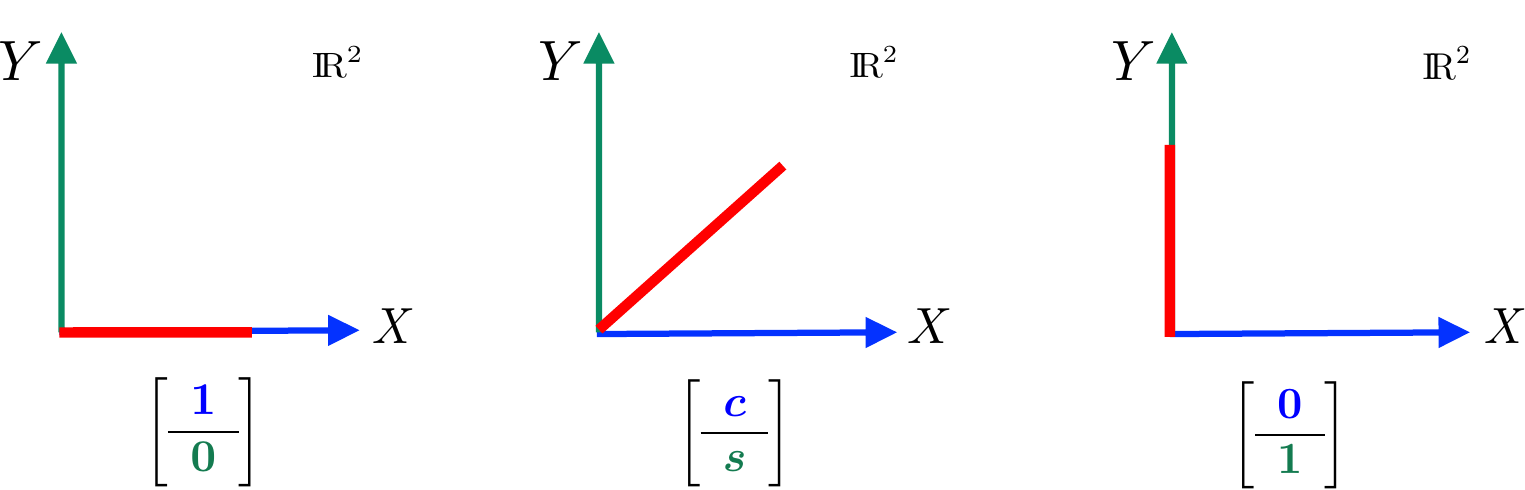}
\end{center}

\subsubsection{\texorpdfstring{A red line in $\R^3$,  $X$=the $xy$-plane,  $Y$=the $z$-axis}{}}

($m_1=2, m_2=n=r=1$ )
Below we show the possibilities for $[C;S]$ for a line in $\R^3$  (drawn in red as the span of $[a,b]$, where $a \in \R^2$, $b \in \R^1$).
The $X$ multiaxis is traditionally labeled the $xy$-plane, and the $Y$ is the $z$-axis.
A line can be in the $xy$-plane, in general position, or along the $z$-axis.  The corresponding $[C;S]$ matrix is illustrated.
The $c$ is the angle between the red line and the $xy$-plane, while the $s$ is the angle of the red line and the $z$-axis.

\begin{center}
\includegraphics[width=.9\textwidth]{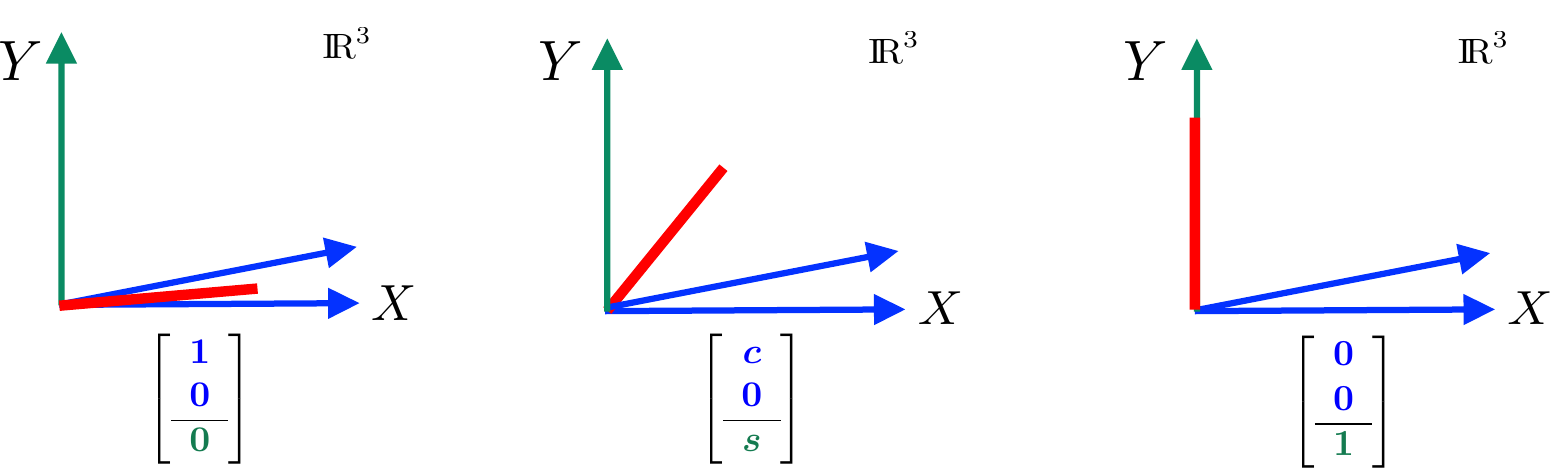}
\end{center}

\subsubsection{\texorpdfstring{A red line in $\R^3$,  $X$=$x$-axis,  $Y$=the $yz$-plane}{}}

($m_1=2, m_2=n=r=1$ )
Below we show the possibilities for $[C;S]$ for a line in $\R^3$  (drawn in red as the span of $[a,b]$, where $a \in \R^1$, $b \in \R^2$).
A line can be along the $x$-axis, in general position, or in the $yz$-plane.  The corresponding $[C;S]$ matrix is illustrated.
The $c$ is the angle between the red line and the $x$ axis, while the $s$ is the angle of the red line and the $yz$-plane.
The shaded $Y$=$yz$-plane indicates the red line is in that plane.

\begin{center}
\includegraphics[width=.9\textwidth]{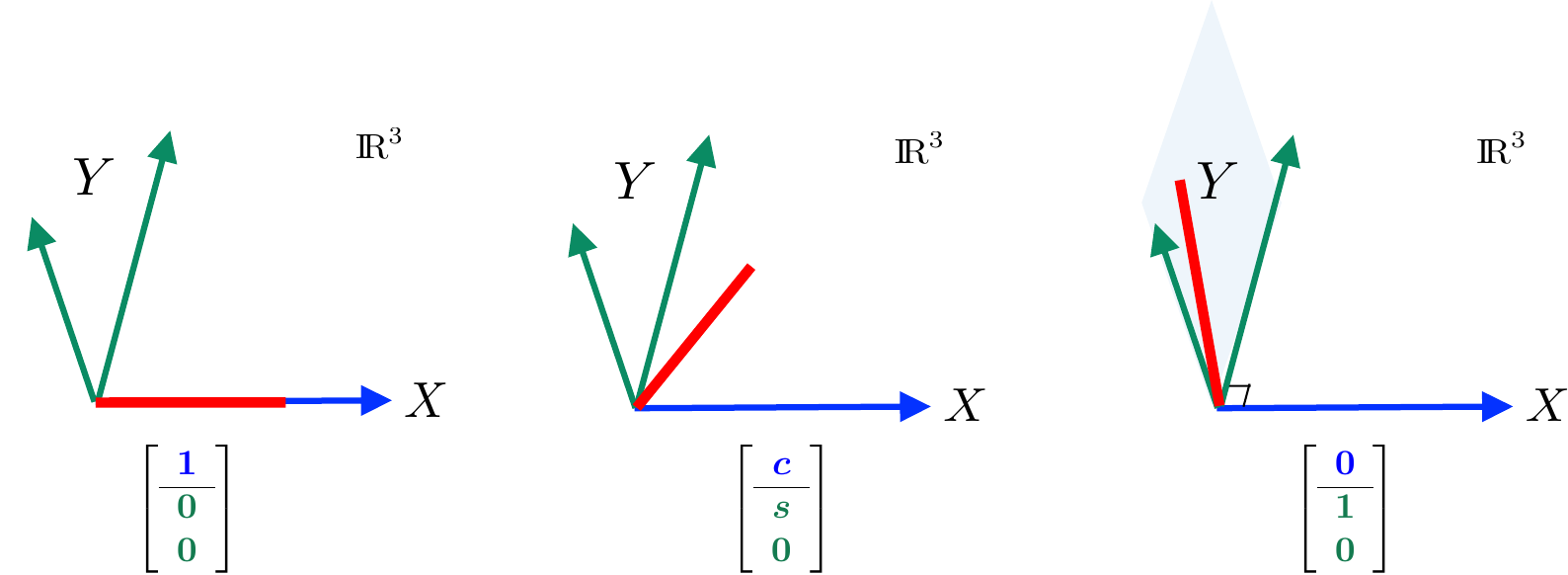}
\end{center}

\subsubsection{\texorpdfstring{A red plane in $\R^3$,  $X$=the $xy$-plane,  $Y$=the $z$-axis}{}}
\label{m212}

($m_1=2, m_2=1, n=r=2$ )
Below we show the possibilities for $[C;S]$ for a plane  in $\R^3$  (drawn in red as the span of $[A,B]$, where $A \in \R^{2,2}$, $B \in \R^{1,2}$).
A plane can be the $xy$-plane.
A plane in general position in $\R^3$  intersects the $xy$-plane in a line (shown as a dashed red line) but does not include the $z$ axis.
A final possibility for a plane is that it includes the $z$ axis (broken red/green line.)
 
 The corresponding $[C;S]$ matrix is illustrated.
 We have $c_1=1$ corresponding to the 0 degree angle  from a line in the red plane and the $x,y$ axis.  We have $c_2$ which is the cosine of the angle formed from a line at right angles from the aforementioned line and the $xy$-plane.  Note that $s_1=0$ is not found in the $S$ matrix, since there is room for only one row which contains $s_2$.

\begin{center}
\includegraphics[width=.9\textwidth]{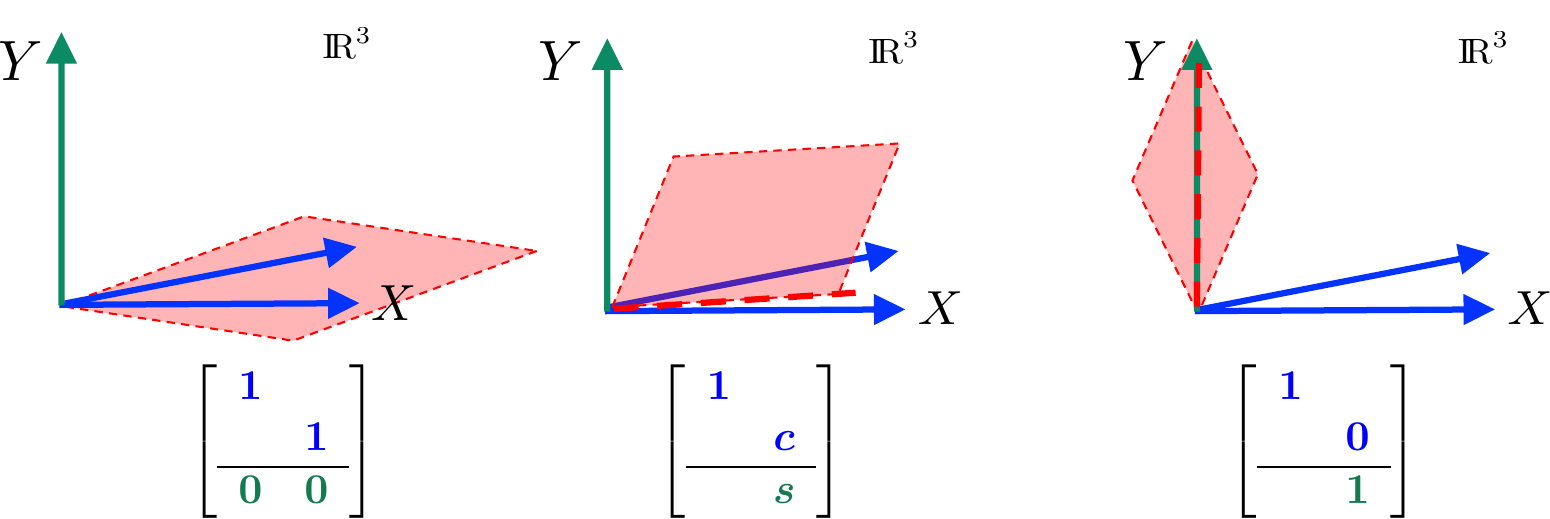}
\end{center}

Figure  \ref{fig:2dexp}
below is the ellipse picture in 3 dimensions (3d), which admittedly has too few dimensions to understand the general picture. Nevertheless, one can clearly see the unit circle in the sphere being projected down to an ellipse on the $x,y$ axis. We see the $c_1=1$ and $c_2=\cos \theta$ as the lengths of the semi-axis of the ellipse.  The $u_1$ direction is where 
the plane representing $\text{span}([A;B])$ intersects the $xy$-plane.  The $u_2$ direction is orthogonal to $u_1$ and also in the $\text{span}([A;B])$ plane.  The $u_2$ direction is the maximum slope 
off the $xy$-plane, and $s_2=\sin \theta$ is the length of the projection of the unit circle onto the $z$-axis.  The orthogonal direction projects to $0$ giving the $s_1=0$. 

 \begin{figure}[htp]
\includegraphics[width=.9\textwidth]{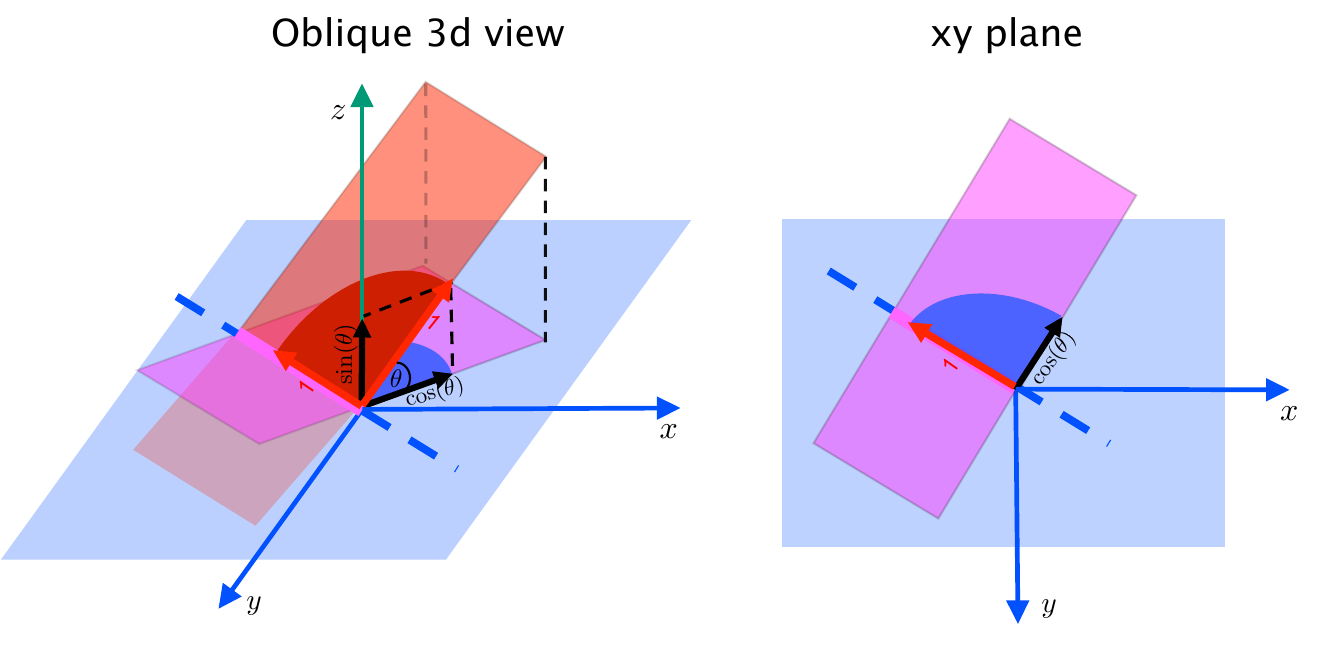}
 \caption{GSVD in 3d is a bit cramped:  Oblique 3d view (left) and $xy$-plane (right). Generically a hyperplane will intersect the $xy$-plane in a line (blue dashed line) which will contain simultaneously the major axis of the blue (cosine) ellipse and a diameter of the  red circle.
In 3d, we have  $c_1=1, c_2=\cos \theta$ to indicate the intersection and the angle $\theta$ with the $xy$-plane, respectively.  We also have $s_1=0, s_2=\sin \theta$ which indicates that with respect to the $z$ axis, the red hyperplane has one vacuous direction  (the red arrow in the $xy$-plane) and the orthogonal direction
(other red arrow in the red hyperplane) makes an angle of $\pi/2-\theta$.
In summary, the blue (cosine) ellipse has semi-axes $1$ and $\cos \theta$, the green (sine) ellipse is confined to 1d and has an unseen 0 and $\sin \theta$, while of course the unit circle has radius $1$.
}

 \label{fig:2dexp}
 \end{figure}

\subsection{On infinite generalized singular values and horizontal directions}
\label{infinite}

As may become clear upon inspection of the small dimensional cases, it is very possible that we have some $c_i=1$ and $s_i=0$
so that the generalized singular value $c_i/s_i$ is infinite.  These infinite singular values are associated with horizontal directions
$[u_i;0]$ in the ``red'' hyperplane, i.e. $[u_i;0] \in {\text span}([A;B])$.  They arise when our hyperplane intersects our $X$ multiaxis
in any non-zero direction.

The situation in Section \ref{m212} illustrates that this is typical when we consider a plane in $\R^3$ and $X$ is the $xy$-plane.
( $A$ is $2 \times 2$ and $B$ is $1 \times 2$.) 
The unit circle in the plane has a  vector of length 1, $[u_1;0]$, that lives on
the horizontal $xy$-plane.  The orthogonal direction, $[c_2 u_2, s_2]$ has a projection $[c_2 u_2;0]$ on the $xy$-plane that is generically
shorter than a unit vector, but still orthogonal to $[u_1;0]$.  


\section{Matrix Trigonometry}
\label{sec:main}
We claim that the GSVD is the natural generalization of high school trigonometry to what we might call ``matrix trigonometry.''


There is so much in Figure \ref{fig:triangle}  that we are all familiar with in the planar case:
There is all of  {\it trigonometry,} and in particular there
is $\tan\theta$ which has a special role because $B/A$ 
is the {\it  slope} of the line. If  $|B|$ is small relative to $|A|,$
we have a shallow slope, and vice versa.  The only hint
that there is some directionality is the  possibility of a $\pm$ sign. 
To specify
directions we sometimes would write a hypotenuse vector in 
{\it component form}: $A\mbox{\bf i}+B\mbox{\bf j}$ .   If we take the components of a unit vector in the direction of the hypotenuse,
then the components form a {\it cosine-sine} pair:  $\cos\theta \mbox {\bf i} + \sin\theta \mbox{\bf j}$.

\begin{figure}[t]
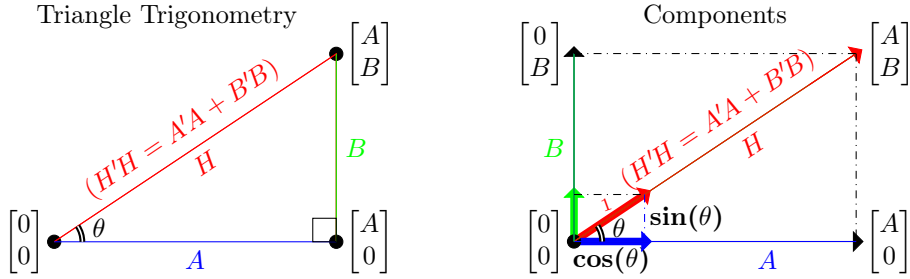

\centering
\drawTriangleA
\hspace{0.4in}
\drawTriangleB
\caption{The GSVD is the generalization of the  trigonometry picture 
(left) or the components picture (right)  to higher dimensions.
When $A$ and $B$ are
$1 , 1$ these pictures specialize
to familiar grade school trigonometry
 (the 2d case where small letters $a$ and $b$ could be used, but we want the reader to think matrix trigonometry as quickly as possible so we will use the capital letters here ).
\newline
\hspace*{.3in} As a portrayer of higher dimensions, line segments represent hyperplanes,
and the desired ellipses are hiding inside the subspaces as the thick unit vector along the hypotenuse (unit sphere in higher dimensions), and the thick components in the cosine-sine pair (horizontal and vertical ellipses in higher dimensions).
\newline
\hspace*{.3in}
Notice that the generalized hypotenuse $H$ is not the matrix square root  but does satisfy $H'H = A'\!A + B'\!B$ (The reason a simple matrix sqrt does not work is that we must denote the direction
of every component in higher dimensions). The cosine form of the GSVD denotes the singular values of $A/H$, and the
sine form denotes  the singular values of $B/H$.
}
\label{fig:triangle}
\end{figure}

The ideas of trigonometry, slope,  component form 
and cosine-sine pairs extend to higher dimensions through the GSVD.  
Instead of one triangle, there are $n$ triangles.  Instead of one vector
\mbox{\bf i}, there are $n$ vectors in the columns of $U$.  Instead of one vector
\mbox{\bf j}, there are $n$ vectors in the columns of $V$.  Instead of a unit
length hypotenuse there are $n$ unit length hypotenuses, which can be written in the component
form 
\[
\cos\theta_k  \begin{bmatrix} u_k \\ 0  \end{bmatrix} + \sin\theta_k  \begin{bmatrix}   0 \\  v_k \end{bmatrix}, \qquad k=1,2,\ldots,n.
\]

The $n$ hypotenuses, as we show in Figure~\ref{fig:elip}, live on a unit sphere that projects nicely ``down''ward and  ``left''ward.  The $\cos\theta_ku_k$  are semi-axes of the downward ellipse; and the $\sin\theta_k v_k$ on the leftward ellipse.

Just as  $b/a$ tells you how small or big $b$ is relative to $a$, the GSVD tells you how small or big $B$ is relative to $A$, but now it is in $n$ natural directions. Thus $B$ can be larger than $A$ in some directions, and smaller in others.

There is some temptation to try to say that the GSVD is related to the
principal angles of the column space of $A$ and the column space of $B$.  This of course
makes no more sense than looking for anything other than right angles between the
$x$-axis and the $y$-axis in 2d.  The interesting angles are between the span of 
the column space of $[A;B]$ and the canonical axes $[I_1; 0]$. More details can be found in Section~\ref{sec:pa}.

One quick algebraic way to define the singular values of an $m, n$ matrix $A$ is to find the diagonal matrix
with non-negative entries in the set $\{UAV' \}$  where $U$ is $m$ by $m$ orthogonal and $V$ is $n$ by $n$ orthogonal.
This is the equivalence class representative definition.
Similarly, one can define the generalized singular values of a pair of matrices $(A,B)$ with the same number of columns.  The ``cosine-sine'' format, is the pair of (1-)diagonal matrices $(C,S)$
with non-negative entries in the set of matrix pairs  $\{(UAH^{-1},VBH^{-1}) : U,V \mbox{ orthogonal}, H \mbox{ non-singular}\}$.
Often the GSVD is given in ``cotangent'' format, which is the ratio of cosines to sines. 

We summarize the GSVD properties with Table~\ref{tab:gsvdp}.

\begin{table}[t]
\begin{center}
    \begin{tabular}{ m{0.7cm} | m{5.5cm}| m{5cm}}
    \toprule
    \multicolumn{3}{c}{}\\
       \multicolumn{3}{c}{$
       \left[
		\begin{array}{c}
		A \\ 
		B
		\end{array}\right] = \left[
		\begin{array}{cc}
		UC \\ 
		 VS 
		\end{array}\right]
		 H$}\\    
       \multicolumn{3}{c}{}\\
       \midrule
       & \\
       $C, S$ &\drawTriangleA  & $\theta$: Principal angle between $\text{span}
    \left\{
       \begin{bmatrix}A \\ B\end{bmatrix}
    \right\}
       $ and $\text{span
       }
  \left\{
       \begin{bmatrix}I_n \\ 0\end{bmatrix}
      \right\}
       $ 
       
       \vspace{.15in}

       $\sin\theta$: \text{SVD}($BH^\dagger$)

      $\cos\theta$: \text{SVD}($AH^\dagger$)

$\tan\theta$: \text{SVD}($BA^\dagger$) if $r=r_a:=\text{rank}(A)$

         $\textbf{cot}\bm{\theta}$: \text{SVD}($AB^\dagger$)  if $r=r_b:=\text{rank}(B)$ \\
       \midrule
       $U$ & \multicolumn{2}{c}{ left singular vectors of $AH^\dagger$ ( or $AB^\dagger$  if $r=r_b$)}\\
       $V$ & \multicolumn{2}{c}{ left singular vectors of  $BH^\dagger$ ( or $BA^\dagger$ if $r=r_a$)}\\
       \bottomrule
    \end{tabular}
\end{center} 
\caption{A primer of the properties of GSVD.}
\label{tab:gsvdp}
\end{table}

\section{The relationship between the GSVD and the CS Decomposition}
\label{sec:cs}
 It is often written~\cite[Section 8.7.5]{golub2012matrix}  that the GSVD and the CS Decomposition are closely related.  The geometric viewpoint highlights the GSVD and the CS decomposition as rooted in  representations of  points in the Grassmann manifold (linear hyperplanes through the origin) in an  $m = m_1 + m_2$ dimensional space using $[UC;VS]$ as natural coordinates.

The simple notion is that the information  may be thought of as 
\[
\begin{bmatrix}A \\ B\end{bmatrix}  = \hspace*{-.5in}
\underbrace{\begin{bmatrix} UC \\ VS\end{bmatrix}}_{\text{\parbox{15em}{\begin{center}column space as a hyperplane  \\ \vspace{-.05in} (a canonical basis!) \end{center}}}}
\hspace*{-.3in}
  \times \   \underbrace{H}_{\text{\parbox{10em}{\begin{center} \vspace{.14in} Coordinates of $[A;B]$ \\ \vspace{-.06in} in the $[UC;VS]$ basis. \end{center}}}}
\]

This connection is rooted ultimately in the Cartan decomposition of the Grassmann manifold, one of the finitely many classes of  symmetric spaces \cite{helgason2001differential}.
The idea is that certain matrix spaces have a ``KAK''  or compact/abelian/compact decomposition.  The SVD is one example as it is orthogonal/diagonal/orthogonal.  The CS decomposition is another.
This observation may be found in a numerical linear algebra conference  presentation  \cite{edelmanhh2002} and in the quantum computing literature \cite{tucci2005introduction}.

To be sure if $[A;B]$ is already orthogonal then so is $H$.  This constitutes the ``left half'' of the complete CS decomposition.  Thus a GSVD is a ``left half'' of a CS, when $[A;B]$ are orthogonal, and the ``left half'' of a CS is
a GSVD.  One can also have a basis for the orthogonal complement of span($[A;B]$) to get the 
``right half.''   This captures the isomorphism between the Grassmann manifold $\mathcal{G}_{m,n}$ (i.e., $n$-dimensional subspace in $\R^m$) and $\mathcal{G}_{m,m-n}$ (i.e., $(m-n)$-dimensional subspace in $\R^m$).  Thus if one takes the combined SVD's of orthogonal matrices whose spans are orthogonal complements, one has the CS decomposition and vice versa.

Any which way, the mathematical idea underlying all is that there is a fairly canonical representation for generic elements of the Grassmann manifold and a matrix connecting back to an orthogonal or arbitrary basis
which has a further symmetry property when taking both the span of $[A;B]$ and its orthogonal complement in conjunction in that transposing a full orthogonal matrix  reverses the roles canonical coordinates and basis converter.

\subsection*{Parameter Count}

There has been a longstanding tradition in numerical linear algebra to overwrite  matrix inputs  with the parameters from the factored form.  Thus if $A$ is $n \times n$, the $LU$ factorization has the $n(n-1)/2$ parameters from $L$ and the $n(n+1)/2$ parameters from $U$.  Similarly if $A=QR$, the $Q$ while appearing naively as an $n \times n$ matrix, actually only has $n(n-1)/2$ parameters, which is exactly
what is computed in software~\cite{anderson1999lapack}.  

Given an $m \times n$ matrix  $[A;B]$ of rank $r$, and a decomposition of $m$ as $m=m_1+m_2$, we can count parameters on both the left and right sides of
$[A;B]=[UC;VS]H.$  While tricky, the only facts needed are:
\begin{enumerate}
\item{Rank Codimension: The codimension of the rank $r$ matrices of size $m \times n$ is $(m-r)(n-r)$  \cite[Lemma 3.3]{demmel1995dimension}.}
\item{Stiefel Manifold Dimension: The dimension of the Stiefel manifold $\mathcal{V}_{m,n}$ of  $n$ ordered orthonormal directions in $\R^m$ is $n(m-n)+ n(n-1)/2$ \cite[Section 2.2]{edelman1998geometry}.}
\item{Grassmann Manifold Dimension: The dimension of the Grassmann manifold $\mathcal{G}_{m,n}$ of  $n$-dimensional subspaces  in $\R^m$ is $n(m-n)$ \cite[Section 2.5]{edelman1998geometry}.}
\end{enumerate}

\vspace{.1in}

\begin{tabular}{|l|lll|} \toprule
& $r \le m_1 \le m_2$ & $m_1 \le r \le m_2$ & $m_1 \le m_2 \le r$ \\ \midrule
rank $r$ codim &{$(m-r)(n-r)$}  & {$(m-r)(n-r)$} & {$(m-r)(n-r)$}\\
$H$ ($r \times n$)& $rn$ & $rn$ & $rn$\\ 
$ 0 < \theta_i < \pi/2$ &  $r$ & $m_1$ & $m-r$ \\ \midrule
\multirow{2}{*}{$U$ Stiefel} & $(m_1 - r)r $ & $m_1(m_1 - 1)/2 $ &  $(r - m_2)(m - r) $   \\
&$+ r(r - 1)/2$&&$ +(m- r)(m - r - 1)/2$ \\
\multirow{2}{*}{$V$ Stiefel} & $(m_2 - r)r $& $(m_2 - m_1)m_1$& $(r - m_1)(m - r)$ \\
 & $+ r(r - 1)/2 $&$ + m_1(m_1 - 1)/2$& $ +(m- r)(m - r - 1)/2$\\
$V$  Grassmann& 0 & $(r - m_1)(m_2 - r)$  & 0 \\ \midrule
Total & $mn$ & $mn$ & $mn$ \\ \bottomrule
\end{tabular}

\vspace{.1in}

To understand the parameter count, we begin with the simple observation that $r_a=\min(r,m_1)$ generically and $r_b=\min(r,m_2)$, from which we can derive the number of $\theta_i$ that
are strictly between $0$ and $\pi/2$ as $r_a+r_b-r$. The relevant Stiefel manifolds are $\mathcal{V}_{m_1,r_a+r_b-r}$ and  $\mathcal{V}_{m_2,r_a+r_b-r}$.  These correspond exactly to choosing the directions
for the axes of the ellipses. Also one must consider $\mathcal{G}_{m_i - (r_a+r_b-r) , r-r_a}$ for $i=1,2$ as this is the dimension divide between the $0$ degree angles and the $\pi/2$ angles when this has content. This data is summarized below:

\vspace{.1in}
\begin{center}
\begin{tabular}{|l|lll|} \toprule
& $r \le m_1 \le m_2$ & $m_1 \le r \le m_2$ & $m_1 \le m_2 \le r$ \\ \midrule
$r_a$ &  $r$  & $m_1$ & $m_1$  \\
$r_b$ &  $r$  & $r$ & $m_2$   \\
$r_a+r_b-r$ & $r$ & $m_1$ &  $m-r$ \\ \midrule
$U$ Stiefel  &  $\mathcal{V}_{m_1,r}$ & $\mathcal{V}_{m_1,m_1}$   & $\mathcal{V}_{m_1,m-r}$ \\
$V$ Stiefel &  $\mathcal{V}_{m_2,r}$ & $\mathcal{V}_{m_2,m_1}$   & $\mathcal{V}_{m_2,m-r}$ \\
$V$  Grassmann & - & $\mathcal{G}_{m_2-m_1,r-m_1}$ & -\\
 \bottomrule
\end{tabular}
\end{center}

\vspace{.1in}

We remark that further fine grain detailed parameter counts are possible including lower rank $A$ and $B$, but we content ourselves with the table above.

\section{Principal angles between subspaces}
\label{sec:pa}

Section~\ref{sec:main} points out that the GSVD of $A$ and $B$ does not contain angle information between the column spaces of $A$ and $B$. Rather, Figure~\ref{fig:elip} illustrates that the relevant angles are between the ``red space'' ($\text{col}([A;B])$) and the ``blue space'' ($\text{col}([I_1;0])$). 

This suggests that the GSVD can be used to compute principal angles (see Section 6.4.3. of~\cite{golub2012matrix}) between the column spaces of $A$ and $B$ when $m_1 = m_2.$ More precisely, it can be accomplished by letting $Z = [Y|Y^\perp]$ be any orthogonal matrix where $\text{col}(Y)=\text{col}(B).$ It follows that $\text{GSVD}(Y'A, (Y^\perp)'A)$ are the cotangents of the desired principal angles. 

This maybe seen geometrically as the GSVD computes the cotangents of angles between 
\[
\text{col}\left(\begin{bmatrix}Y'A \\ (Y^\perp)'A\end{bmatrix}\right) = \text{col}(Z'A), \quad \text{and},\quad \text{col}\left(\begin{bmatrix}I_1\\0\end{bmatrix}\right),
\]
but we can multiply by the orthogonal matrix $Z$, which preserves angles, obtaining the angles between $\text{col}(A)$ and $\text{col}(Z[I_1;0]) = \text{col}(B).$

We can conclude that we have a rotated Figure~\ref{fig:elip} (shown in Figure~\ref{fig:rotate}) where the $X$ and $Y$ multiaxes are replaced with $\text{span}(Y)$ and $\text{span}(Y^\perp).$

\section{The Lemniscate Plots from Leuven, Belgium}
\label{sec:energy}

In a series of early papers most of which date back to the 1980s \cite{callaerts1989signal,callaerts1990comparison,chu2003qr,de1988mathematical,de1988oriented,vandewalle2003generalized}, energy portraits 
that relate to the SVD and GSVD of a matrix or a pair of matrices are discussed with applications.

The definition of an energy portrait of a single matrix is
\[\mbox{Energy}(A) = \{  e\|Ae\|^2 : \|e\|=1 \} \subset \R^n, \ (A\in \R^{m,n})\]
and for a pair of matrices with the same number of columns
\[\mbox{Energy}(A,B)=\left\{ e \frac{  \|Ae\|^2 }{\|Be\|^2} : \|e\|=1 \right\} \subset \R^n, \ (A\in \R^{m_1,n}, B\in\R^{m_2,n}).\]

\begin{figure}[ht]
    \begin{minipage}{.5\textwidth}
        \centering
        \includegraphics[width=\textwidth]{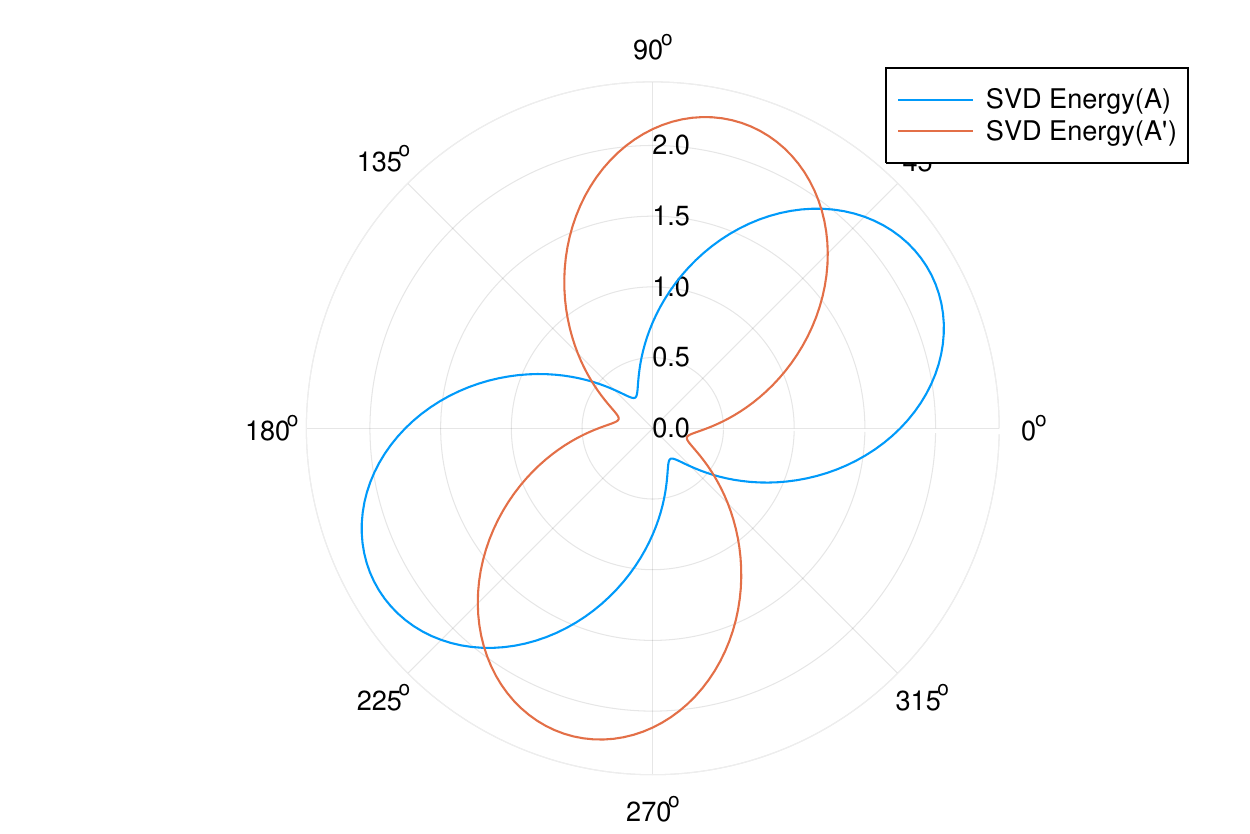}        
    \end{minipage}%
    \begin{minipage}{0.5\textwidth}
        \centering
        \includegraphics[width=\textwidth]{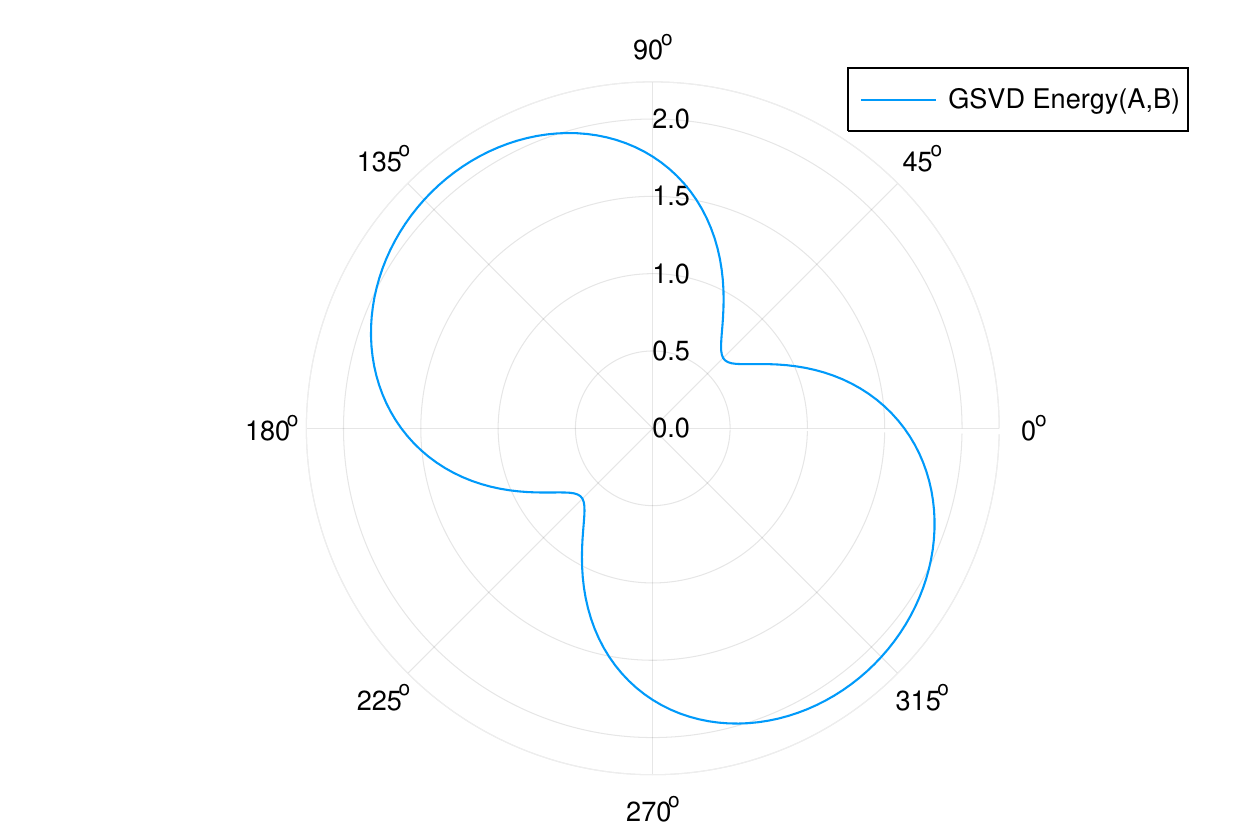}              
    \end{minipage}
  \caption{Leminiscate plots:  Energy($A$) and Energy($A'$) reproducing    from \cite[Figure 3]{staar1982singular} (left) and using the matrices below for Energy$(A, B)$ (right).}
\label{fig:lemin}
\end{figure}

It is important to point out that the curves in Figure~\ref{fig:lemin} are not ellipses but rather lemniscate-like portraits.
They do not even live in the same spaces as the ellipse pictures.  
The standard SVD ellipse lives in  $\R^m$ and the GSVD picture in this paper lives in $\R^{m_1+m_2}$.
By contrast, the energy portraits from Leuven live in $\R^n$.

We provide the Julia codes that produce these curves as a reference.  Readers are encouraged to try other matrices. \\

{
\hspace{.1in}
\begin{jinput}
\small
\begin{verbatim}
A = [.577699 -.224144;1.190069 .836516] # Figure 6 (Left)
e(theta) = [cos(theta), sin(theta)]
r1(theta) = sum(abs2, A*e(theta))
r2(theta) = sum(abs2, A'e(theta))
theta = pi * (0:.01:2)
plot( theta, r1.(theta), proj=:polar, label="SVD Energy(A)")
plot!(theta, r2.(theta), proj=:polar, label="SVD Energy(A')") 
\end{verbatim}
\end{jinput}
}

{
\hspace{.1in}
\begin{jinput}
\small
\begin{verbatim}
A = [.27 .66 ; -1.4 1.3] # Figure 6 (Right)
B = [1 0; -.5 1.1]
e(theta) = [cos(theta), sin(theta)]
r1(theta) = sum(abs2, A*e(theta))
r2(theta) = sum(abs2, B*e(theta))
theta = pi * (0:.01:2)
plot(theta,r1.(theta)./r2.(theta),
     proj=:polar,label="GSVD Energy(A,B)") 
\end{verbatim}
\end{jinput}
}

For completeness, we thought we would take a closer look at these older plots.
To explain in what sense the curves are lemniscates, it is best to eliminate the ``e'' in the definition
and rewrite the energy plots as the zero set of an algebraic equation, thereby connecting
the portraits to the field of  algebraic geometry.

\begin{theorem}
If $Vx \ \in \mbox{Energy}(A)$, then $x$ satisfies the algebraic polynomial equation
$$ \left[\sum x_i^2\right]^3 = \left[\sum \sigma_i^2 x_i^2\right]^2,$$
where $A=U\Sigma V'$. Further if 
$x  \in \mbox{Energy}(A,B)$, then $x$ satisfies the algebraic polynomial equation
$$\|x\|^2\|SHx\|^4 = \|CHx\|^4,$$
where $[A;B]=[UC;VS]H$.
\end{theorem}

Before proving the theorem we provide a historical analog.
 We might compare  the solution set of
$ (\sum_{i=1}^n x_i^2)^3 = (\sum_{i=1}^n \sigma_i^2 x_i^2)^2,$ with that of
$ (\sum_{i=1}^2 x_i^2) = (\sum_{i=1}^2 \sigma_i^2 x_i^2),$  which is the lemniscate of Booth whose study traces back to the 5th century Greek philosopher Proclus.
The difference being that Booth specialized to $n=2$ and only took first powers of the quantities, but in spirit it is a similar algebraic polynomial equation.

\begin{proof}
Taking $e=Vy$, we see that  $e\|Ae\|^2 = Vy\|\Sigma y\|^2=Vx$ where $x=y\|\Sigma y\|^2.$
It is straightforward to check $\|x\|^6=\|\Sigma x\|^4=\|\Sigma y\|^{12},$ since $\|y\|=1$ which is exactly the result for a single matrix.

For the two matrix case, where $A=UCH$ and $B=VSH$,  if $x=e\|Ae\|^2/\|Be\|^2$, then 
$$\|x\|^2=\frac{\|CHe\|^4}{\|SHe\|^4}, \ \  \mbox{ and} \ \ \frac{\|CHx\|}{\|SHx\|} = \frac{\|CHe\|}{\|SHe\|} . $$
\end{proof}

\section{\texorpdfstring{On the $\text{GSVD}(A, B)$ and the $\text{SVD}(AB^\dagger$)}{}}
\label{sec:thm}
In this section we relate the finite part (nonzero, noninfinite) of the generalized singular values of $(A, B)$ (denoted as $\text{GSVD}(A, B)$) to the singular values of $AB^\dagger$ (denoted as $\text{SVD}(AB^\dagger)$) where $B^\dagger$ is the pseudoinverse of $B$.
We may use the notation $A/B$ for $AB^\dagger$.  An issue arises that may surprise some readers.

\subsection{Why there is an issue?}

One may expect that there may always be a relation between the GSVD of $A,B$ and the SVD of $AB^\dagger$.
For example, in the \textsc{matlab} documentation\footnote{\url{https://www.mathworks.com/help/matlab/ref/GSVD.html}}
it is stated that the generalized singular values are the ratios of the 
diagonal elements of $C$ and $S$ in a given example.
One might infer from the documentation that this is always the case.

However it is not generally true when there are infinite singular values, i.e., when $r_b < r$.

Consider a simple example where $A$ is  a non-singular $n \times n$ matrix, and $B$ is a nonzero $1 \times n$ matrix.
In this case $r_b=1,r=n$.
The GSVD of $A,B$ is readily verified to have $n-1$ infinite singular values, and the one finite value
 $\sigma_{\text{GSVD}} = 1/\|B/A\|.$ The SVD of $AB^\dagger$ is just the length of $AB^\dagger = AB'/\|B\|^2$ or
$\sigma_{\text{SVD}} = \|BA'\|/\|B\|^2.$

When $n=1, A=a,B=b$, both of these expressions are equal to the absolute ratio $|a/b|$,  ($r=r_b=1$ after all) but for larger $n$
 the two matrix expressions  are not equal.

An extremely simple special case takes $A = \begin{pmatrix} 3 & 0 \\ 0 & 4 \end{pmatrix}$ and $B=(1 \ \ 1).$
The two values are   $ \sigma_{\text{GSVD}} =2.4 $ and $\sigma_{\text{SVD}}=2.5$ exactly.

The issue arises exactly when there are infinite $\sigma$.
If there are no infinite $\sigma$, $S$ has no $0$ columns, and we can write
$$AB^\dagger = (UCH)(VSH)^\dagger= UCHH^\dagger S^\dagger V' = U(C/S)V',$$
which is a singular value decomposition of $A/B$.  (We use the property that $H$ has full row rank to conclude $HH^\dagger=I_r$
and that $C/S$ is an $m_1 \times m_2$ matrix with $c_i/s_i$ on the main diagonal.)

The problem that arises when some $\sigma=\infty$ is  that  $B^\dagger=(VSH)^\dagger=(SH)^\dagger V'$ does not equal $H^\dagger S^\dagger V'$
when $S$ has any zero columns.

\subsection{\texorpdfstring{The significance of horizontal directions and their orthogonal complement in $X$}{}}
\label{horizontal}

In Section \ref{infinite}, we considered the intersection of span($[A;B]$) with the $X$ multiaxis.
An orthogonal basis for this intersection is $[u_1;0],\ldots,[u_{r-r_b};0]$ which correspond exactly
to the $c_i=1$.

Working entirely in $X$ as an $m_1$ dimensional space, we are interested in the $m_1 \times m_1$ projection
matrix $P$ that kills the directions of intersection.  Precisely we define $P$ on the orthogonal basis for $\R^{m_1}$:
$$
Pu_i =
  \begin{cases}
    u_i  & \text{if $c_i < 1$} \\
    0 & \text{if $c_i =  1$}.
  \end{cases}
$$

Suppose $N$ is a matrix whose columns are a basis for the null space of $B$.  If we consider $AN$ then 
the span of the  columns of $AN$ is the intersection
we are discussing, i.e., the intersection of $X$ with span($[A;B]$).
To be sure either the column of $N$ is in the common null space of $A$ and $B$, so that the corresponding column of $AN$ is $0$,
or else if one follows through the first $r-r_b$ columns of $H^\dagger$ in $A=UCH^\dagger$, one sees that we will
hit the ``$c_i=1$'' columns in $C$ only, hence we will emerge a linear combination of $u_1,\ldots,u_{r-r_b}$.

We can thus describe $P$ as the orthogonal projection onto the left nullspace of $AN$ which is the orthogonal complement of the column space of $AN$.

\subsection{\texorpdfstring{The correct modified theorem requires $PA/B$}{}}
We remind the reader of the usual definition of the matrix pseudoinverse in terms of the singular value decomposition:
\begin{equation}
A^\dagger = V\Sigma^\dagger U', 
\end{equation}
where $\Sigma^\dagger$ means taking the inverse of the finite entries in $\Sigma.$ When $A$ has full column rank and $B$ has full row rank, we have $(AB)^\dagger = B^\dagger A^\dagger.$ It is easy to see that $[\bm{0}\ B]^\dagger = [\bm{0}; B^\dagger].$

\begin{theorem}
\label{thm:pthm2}
Let $N$ be a  matrix whose columns are a basis for the nullspace of $B$, and  $P$  be the orthogonal projection onto the left nullspace of $AN$.
The finite non-zero generalized singular values of $(A, B)$ are the same as the non-zero singular values of $PAB^\dagger$.
\end{theorem}

\begin{proof}

\newcommand\bovermat[2]{%
  \makebox[0pt][l]{$\smash{\overbrace{\phantom{%
    \begin{matrix}#2\end{matrix}}}^{\text{#1}}}$}#2}

Setting notation, we have

\begin{align*}
\overbrace{\hspace*{.45in}}^{C_*} \hspace{.3in} &  \\[-.05in]
 A = U   \left[ \begin{array}{c|cccc}  
\hspace*{-.05in}
\renewcommand{\arraystretch}{.1}
\setlength{\arraycolsep}{.1pt}
\begin{array}{ccccc}  
{\text  {\tiny 1}} \\ 
&  {\text  {\tiny 1}} \\ 
&&  .\\ \\
&&&  .\\
&&&&   {\text  {\tiny 1 \hspace*{-.1in}} }
 \end{array} & 0
 \\ \hline
& \renewcommand{\arraystretch}{.1}
\setlength{\arraycolsep}{.1pt}
\begin{array}{ccccc}
{\text  {\tiny $\! \! \! \!  c_{r-r_b+1}$}} \\ 
\ \  \ \ \  \vspace{.2in} \ddots \\ \\
 \end{array}
  \renewcommand{\arraystretch}{1.0} 
\end{array}\right]
H 
\\
\overbrace{\hspace*{.45in}}^{S_*} \hspace{.3in} &  \\[-.25in]
B = V  \left[ 
\begin{array}{c|cccc}  
\hspace*{-.05in}
\renewcommand{\arraystretch}{.1}
\setlength{\arraycolsep}{.1pt}
\begin{array}{ccccc}  
{\text  {\tiny 0}} \\ 
&  {\text  {\tiny 0}} \\ 
&&  .\\ \\
&&&  .\\
&&&&   {\text  {\tiny 0 \hspace*{-.1in}} }
 \end{array} & 0
 \\ \hline
& \renewcommand{\arraystretch}{.1}
\setlength{\arraycolsep}{.1pt}
\begin{array}{ccccc}   \\
{\text  {\tiny $\! \! \! \! \!  \! \! \! s_{r-r_b+1}$}} \\ 
\ \  \ \ \  \vspace{.2in} \ddots \\ \\
 \end{array}
  \renewcommand{\arraystretch}{1.0} 
\end{array}\right] H  & = V
\overbrace{
 \left[ \begin{array}{cccc}  
\hspace*{-.05in}
0
 \\ \hline
 \renewcommand{\arraystretch}{.1}
\setlength{\arraycolsep}{.1pt}
\begin{array}{ccccc}   \\
{\text  {\tiny  $\! \!\! \! \! \! \! \!  \! \! \! \!  s_{ {}_{r-r_b+1}}$}} \\ 
\ \  \ \ \ \ \ \vspace{.2in} \ddots \\ \\
 \end{array}
  \renewcommand{\arraystretch}{1.0} 
\end{array}\right]
}
^{S_*}
H_* = VS_*H_*, 
\end{align*}
so that $B=VS_*H_*$, where $S_*$ are the rightmost $r_b$ non-zero columns of $S$ (indexed by $i=r-r_b+1,...,r$) and $H_*$ are the corresponding rows (the bottom $r_b$)
of $H$. (To see this note that $B=V[0 \ S_*][?;H_*]$ where the ``?'' denotes rows that hit the 0 columns in $S$ so we do not care what they are.)
We point out that $H_*$ has full row rank as the rows of $H_*$ are a subset of the full row rank matrix $H$. We immediately conclude that
\[ B^\dagger = H_*^\dagger S_*^\dagger V'.\]

We further claim that
\[ PA = UC_*H_*, \]
where $C_*$ are the exact corresponding columns of $C$ (the rightmost $r_b$ indexed by $i=r-r_b+1,...,r$), which are the $c_i<1$.
To see this, first observe that the definition of $P$ as described in Section \ref{horizontal}. is $PU=U[0 \ \  I_*]$ where $I_*$ are the rightmost $r_b$ columns of the identity indexed by  $i=r-r_b+1,...,r$.
Thus $PA=U[0 \ C_*][?;H_*]= UC_*H_*$ the $0$ indicating the columns of $U$ killed by $P$.

Now that we have  compressed out the immaterial columns,  and knowing that $H_*H_*^\dagger=I_{r_b}$ by the full row rank condition, we can compute
$$PAB^\dagger  = UC_*H_*H_*^\dagger S_*^\dagger V' = UC_*/S_* V'.$$
This is a singular value decomposition of $PAB^\dagger$, with $\Sigma = C_*/S_*$ an $m_1 \times m_2$ diagonal matrix, with
the $c_i/s_i$ in decreasing order on the diagonal and no $s_i=0$.

\end{proof}

\begin{corollary}
If $B$ has full column rank ($r_b=n$)  or if the weaker condition holds that $r=rank([A;B])=r_b=rank(B)$, then $P$ is not needed, i.e., the finite non-zero generalized singular values of $(A,B)$ are the same as the non-zero singular values of $AB^\dagger$.
\end{corollary}

\proof
If $r_b=n$, then $B$ has nothing in the nullspace, $N$ has no columns, and  $P$ is obviously $I$.  More generally, if $r_b=r$, then $B$ has nothing in its nullspace that is
not also in the nullspace of $A$, so if $AN$ has any columns at all, it  is the zero matrix, so again projection onto the left nullspace is $P=I$.

\subsection{Blame the pseudoinverse not the GSVD}

The difficulty with $AB^\dagger$ may seem like an unfortunate consequence of infinite singular values, but in point of fact, it is related to the discontinuity
in the definition of the pseudoinverse.  If one takes a bigger picture viewpoint, it is easy to see that infinite singular values are natural limits of finite singular values.

The only truly natural discontinuity in the GSVD is the reduction of rank of $[A;B]$ which reduces the dimensionality of the hyperplane (and the rank of $H$.)

We mention some  limit type results which help understand the nature of the infinite generalized singular values:

\begin{theorem}
\label{limittheorem}
If rank($[A;B]$)=$r$, and $m_2 \ge r$, then we can define a continuous curve of matrices $[A_\epsilon,B_\epsilon]$  of the same shape as $[A;B]$
without infinite generalized singular values when $\epsilon>0$ is small but  whose limit as $\epsilon  \rightarrow 0$ continuously
converges to the generalized singular values of $[A,B]$, finite or infinite. 
\end{theorem}

\begin{proof}
Take
$$
\begin{bmatrix}
A_\epsilon\\
B_\epsilon
\end{bmatrix} = \begin{bmatrix}UC(\epsilon)\\VS(\epsilon)\end{bmatrix}H,$$
where 
$$
c_i(\epsilon) = \begin{cases} c_i & s_i > 0 \\ \cos(\epsilon) & s_i=0 \end{cases} \ \ \  \mbox{and } \
s_i(\epsilon) = \begin{cases} s_i & s_i > 0 \\ \sin(\epsilon) & s_i=0 \end{cases}.$$
\end{proof}

\begin{corollary}
\label{limitcor}
If rank($[A;B]$)=$r$, and $m_2 < r$, then we can define a continuous curve of matrices $[A_\epsilon,B_\epsilon]$ 
without infinite generalized singular values when $\epsilon>0$ is small but  whose limit as $\epsilon  \rightarrow 0$ continuously
converges to the generalized singular values of $[A,B]$ by row augmenting $B_\epsilon$ to contain $r$ rows.
\end{corollary}
\begin{proof}
Simply add $r-m_2$ rows of zeros to the bottom of $B$.  This does not change the generalized singular values of $[A;B]$ or $U$,$C$ or $H$.
$S$ is augmented with $r-m_2$ rows of zeros and $V$ is augmented with $r-m_2$ rows and columns with an identity matrix.  
Apply the construction in Theorem \ref{limittheorem} to complete the proof.
\end{proof}

\begin{example}
Consider that
\[
\mbox{\hspace*{-1in} GSVD}
\left(
\begin{bmatrix}
3 & 0 \\
0 & 4
\end{bmatrix},
\begin{bmatrix}
1 & 1 \\
\end{bmatrix}
\right)
= \ 2.4 \mbox{ and } \infty.
\]
One might seek nearby matrices with no infinite generalized singular values.
This is impossible if we insist that $B$ remain $1 \times 2$ but is
possible if we augment $B$ with one row, which in this case we can simply take
\[
\text{GSVD}
\left(
\begin{bmatrix}
3 & 0 \\
0 & 4
\end{bmatrix},
\begin{bmatrix}
1 & 1 \\
0 & \epsilon
\end{bmatrix}
\right)
= \ 2.4 +O(\epsilon^2) \mbox{ and } 5/\epsilon+O(\epsilon). 
\]
\end{example}

\begin{corollary}
Suppose $[A_\epsilon,B_\epsilon]$ has rank $r$ for $0 \le \epsilon < \epsilon_0$ is a continuous curve, where $B_\epsilon$ has rank $r$ for $\epsilon>0$
but may drop rank at $\epsilon=0$.  We then have that the generalized singular values are a continuous function of  $[A_\epsilon,B_\epsilon]$ as $\epsilon \rightarrow 0$.
\end{corollary}

\begin{proof}
The only true discontinuity in the GSVD is the potential for a drop in rank of $[A;B]$.  This is avoided in the statement by keeping  $[A_\epsilon,B_\epsilon]$  rank $r$.
Thus the limit of the column space is the column space of the limit.
\end{proof}

We do remark on the other hand that if $[A_\epsilon,B_\epsilon]$ drops rank, then we can only say that the limit of the column space contains  the column space of the limit, which
can lead to all kind of discontinuities in the generalized singular values.

\section{GSVD Applications and their Geometric Interpretations}
\label{sec:applications}
\subsection{Geometry of Tikhonov Regularization}

\subsubsection{The two cosine damping}

We show how geometry can add insight to our understanding of Tikhonov Regularization:
\begin{equation}
\label{eqn:tr}
  \min_{x} \left\{ \|Ax - b\| + \lambda\cdot\|Lx\| \right\} 
\end{equation} 
by providing a two cosines view of damping.
Specifically,  the way Tikhonov regularization reduces the solution or  ``weights,'' is usually
understood algebraically in terms of adding a regularizer term that moves the original problem away
from some kind of ill-conditioned setting.  We will show that, in Figure~\ref{fig:tik2d}, one cosine comes from the projection
from the horizontal (blue) plane to the span of $[A;\lambda L]$ red plane. The other cosine comes from
the non-canonical basis of the plane: the columns of $[A;\lambda L]$ which elongate with $\lambda$, hence the coordinates shrink.

While the ``calming influence''~\cite[Section 6.1.26]{golub2012matrix},~\cite[Section 4.4]{un1992csd},~\cite{hansen1989regularization}
of the regularization parameter $\lambda$ 
has
been well studied algebraically, we identify geometrically  in (\ref{eqn:xlam}) the influence as a factor of $\cos^2 \theta_\lambda$ where $\tan \theta_\lambda = \lambda \tan \theta_1$ 
so that $\cos^2 \theta_\lambda =  1/(1+\lambda^2  \tan^2 \theta_1)$, where $\theta_1$ is the angle that corresponds to $\lambda = 1.$ 
We will compare the $\cos^2$ formulation with previous formulations explaining why we find that  this formulation feels somewhat more insightful.

Before we start, let us recap Tikhonov regularization. Suppose we have a matrix $A$, which we will assume has full column rank.  
The $\lambda=0$ problem (standard least squares) is the computation of 
$x_0= A^\dagger  b = (A'A)^{-1}A'b, $
the standard solution to the normal equations $A'Ax =A'b.$ To regularize we pick a suitable matrix $L$, and a ``regularization parameter'' $\lambda$,
and then solve instead
$(A'A + \lambda^2 L'L)x=A'b,$
which is equivalent to computing
\[x_\lambda = \left[ \begin{array}{c} A \\ \lambda L  \end{array} \right]  ^\dagger
 \left[ \begin{array}{c} b \\ 0  \end{array} \right].
\]
From the geometrical point of view, we believe the reformulation in Theorem \ref{tiktheorem} below is more revealing of the ``calming effect.'' Figure~\ref{fig:tik2d} demonstrates the hyperplane onto which $[b;0]$ gets projected for varying $\lambda.$

 \begin{figure}[htp]
\includegraphics[width=.9\textwidth]{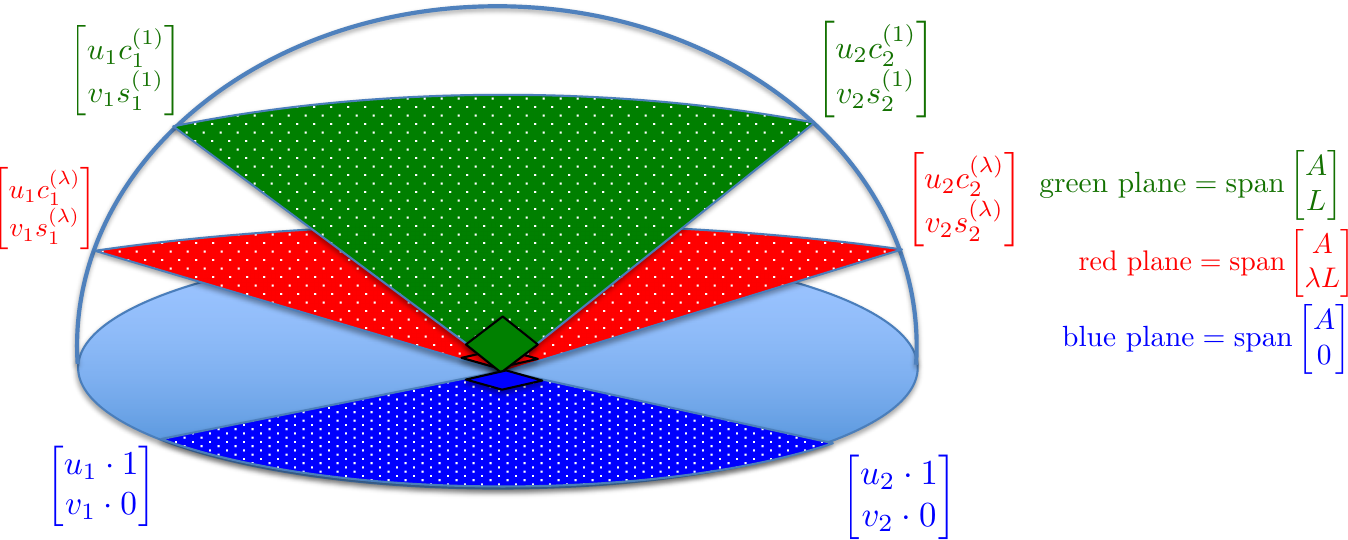}
 \caption{This n=2 Tikhonov regularization  picture in the four dimensional hypersphere illustrates the hyperplanes onto which $[b;0]$ gets projected for varying $\lambda$.  The projection gives one cosine, while the representation (not pictured)  in ever elongating bases gives the second cosine.  Portrayed is the unique hypersphere containing the four mutually orthogonal vectors in four dimensions: $[u_1,0],[u_2,0],[0,v_1],[0,v_2]$,
While tempting to see this as a 3d object, as $\lambda \rightarrow \infty$ the wedge drawn does not shrink but remains a quarter circle wedge.
}
 \label{fig:tik2d}
 \end{figure}

For every $\lambda$, we obtain the GSVD as a continuous function of $\lambda$: 
\[
\begin{bmatrix}
A\\
\lambda L 
\end{bmatrix} = 
\begin{bmatrix}
U C_\lambda \\ 
V S_\lambda 
\end{bmatrix} H_\lambda,
\]
where it is easy to check that $H_\lambda$ is square non-singular.
It is convenient to use the compact format described in Section \ref{compactformat} here.  Thus we take $U$ to be $m_1 \times n$, $C$ and $S$ to be square
diagonal $n \times n$.  The exact values in $C$ and $S$ come from the trigonometry with unit hypotenuse, fixed base, and sliding
height of a $c,s,1$ triangle at $\lambda=1$,  as shown in the left side of Figure \ref{fig:tik}. Namely
\[C_\lambda = \frac{C_1}{\sqrt{C_1^2+\lambda^2 S_1^2}} \ \ \mbox{ and } \ \ S_\lambda = \frac{\lambda S_1}{\sqrt{C_1^2+\lambda^2 S_1^2}},\]
where the operations happen on the diagonal. It also follows that 
\[H_0=C_\lambda H_\lambda,\ \text{and}, \ A=UH_0 = UC_\lambda H_\lambda,\ \forall\lambda\ge 0.\]

The equation $H_0 = C_\lambda H_\lambda$  has a nice trigonometric interpretation.  As the column vectors of $[A;\lambda L]$ grow in length
(these lengths are encoded in $H_\lambda$). the cosines in $C_\lambda$ relate back to the $[A;0]$ columns which are shorter in length. This is depicted
in Figure \ref{fig:tik}.

\begin{theorem}
\label{tiktheorem}
The solution $x_\lambda$ to the Tikhonov Regularization problem  can be written as 
\begin{equation}
\label{eqn:xlam}
x_\lambda = \left(  H_0^{-1} C_\lambda^2 \ H_0  \right) x_0, 
\end{equation}
where $x_0$ is the least squares solution to $Ax=b$ and $A=UH_0$, where $[A;\lambda L]=[UC_\lambda;VS_\lambda]H_\lambda.$
\end{theorem}

\begin{proof}
Since
$$x_\lambda = \left[ \begin{array}{c} A \\ \lambda L  \end{array} \right]  ^\dagger
 \left[ \begin{array}{c} A \\ 0  \end{array} \right]  x_0,$$
 we can calculate
 
 $$x_\lambda = H_\lambda^{-1}C_\lambda U' U H_0 x_0 = H_\lambda^{-1}C_\lambda H_0 x_0$$
 and use the relation
 $H_\lambda^{-1} = H_0^{-1}C_\lambda$ to complete the proof.
\end{proof}

\begin{figure}[ht]
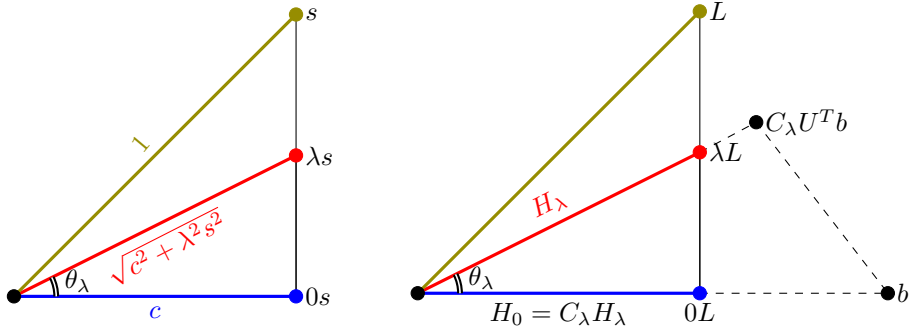

\begin{minipage}{.35\textwidth}

        \centering
        \drawTikonov
    \end{minipage}%
\begin{minipage}{.65\textwidth}
        \centering
        \drawTikonovGeneral
    \end{minipage}%
    \caption{The ``two cosine'' Geometric interpretation of Tikhonov regularization:  Single u-v plane (left) vs. general (right). The green, red and blue lines represent the span of $[A; L]$ (green) , $[A;\lambda L]$ (red) and $[A;0L]$  (blue) respectively.  Our ``two cosines'' view of regularization is that one cosine dampening comes from the projection of $b$ from the blue plane to the red plane, and the second
    cosine comes from the extended basis $H_\lambda = C_\lambda^{-1} H_0$ which gets divided.  
  Note that the value of $\lambda$ may be greater than 1  (not shown).}
\label{fig:tik}
\end{figure}

\paragraph{Comparison and Discussion}

The standard application of the GSVD to Tikhonov relates $x_\lambda$ to $b$ and thus gives formulas involving the non-physical, non-homogeneous factor of
$c/(c^2+\lambda^2 s^2)$ rather than the homogeneous $c_\lambda^2 = c^2/(c^2+\lambda^2 s^2)$.

The formulation in  Theorem  \ref{tiktheorem} diagonalizes the operator that relates  $x_\lambda$ to $x_0.$ We understand that when $x$ are the coordinates of a linear combination of the columns of $[A;B]$,
we have that $H_0x$ are the coordinates of that same vector in the natural basis.  Thus the interpretation of
$H_0^{-1}C_\lambda^2H_0$ simply is:
\begin{enumerate}
\item Write the vector in the natural coordinate system;
\item Multiply by a cosine squared in every natural direction;
\item Return to the original coordinate sytem.
\end{enumerate}

\subsection{Humans vs Yeast: Comparative Data Modeling}

In a series of beautiful applications of the GSVD, Alter, et.al.\ \cite{alter2003generalized,ponnapalli2006higher,ponnapalli2011higher,sankaranarayanan2015tensor,aiello2018mathematically} propose an approach towards data reconstruction
and classification.  In their case~\cite{alter2003generalized}, the $A$ and $B$ are two DNA microarrays, one from humans and the other from yeast.
The rows of $A$ and $B$ live in $\R^n$ or gene space.  The rows of $H$ form a basis for this row (or gene) space, and are denoted
genelets.  A natural question is whether the genelet is primarily human, primarily yeast, or a mixture.  In general, given two matrices
with equal columns, one wants to classify the basis vectors in the rows of $H$ according to its source.

The GSVD provides a natural solution
by creating  a single coherent model from the two datasets recording different aspects of interrelated phenomena by simultaneously identifying the similar and dissimilar between the two corresponding column-matched but row-independent matrices.
 For each of the $r$ rows, we have that  $\theta_i$ denotes the angle towards $A$.
In Figure \ref{fig:pi5}, we portray this.  We  note that \cite{alter2003generalized} displays the angles from $-\pi/4$ to $\pi/4$, but we will stick with the $0$ to $\pi/2$ convention.
It is convenient that the rows of $H$ are already sorted from ``mostly $A$,'' to ``mostly $B$.''  

Our ellipse picture Figure \ref{fig:elip} reveals the geometry readily.  The $[u_ic_i;v_is_i]$ all appear on the unit ball.

The comparative Data Reconstruction equation is
$$\begin{bmatrix} A \\ B \end{bmatrix} =  \sum_{i=1}^r \begin{bmatrix}  u_i c_i \\ v_i s_i \end{bmatrix} h_i',$$
where $h_i'$ is the $i$-th row of $H$. (This is exactly Equation (\ref{gheq}).)
One can preprocess $H$ so that each row is of unit direction as it is only the ratio of $c_i$ to $s_i$ that matters.
Any  ill-conditioning of $H$ could be worrisome.

\begin{figure}[ht]
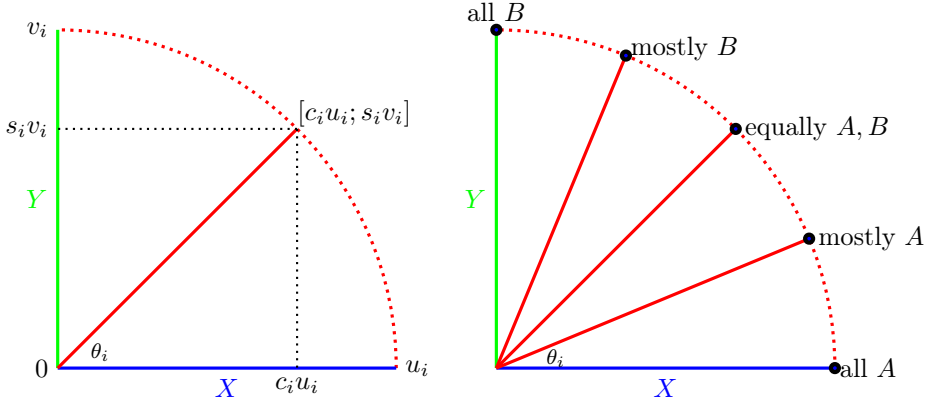

    \centering
    \drawGraph        
  \drawAnglePlot     
  \caption{
Are genomes human, yeast, or a combination?   (Application from Alter  \cite{alter2003generalized})
Left: $[c_iu_i;s_iv_i]$ makes an angle $\theta_i$  with the $X$ multiaxis.  Right: Depending on the angle we apportion the $i$th row of $H$ ( a basis element for the row spaces of $A$ and $B$) as being attributable to $A$ or $B$. }
\label{fig:pi5}
\end{figure}

\subsection{Signal vs. Noise: A one matrix and one subspace view of the GSVD}
\label{onematrix}
The focus on two matrices with the same number of columns is not always the best view of the GSVD.  One can take rather a single
$m \times n$ matrix $M$ and any $m_1$ dimensional reference subspace ${\cal S}$ of $\R^m$.  We can then think of the GSVD as an
additive decomposition:
$$M=P+Q,$$
where $P=Y_1UCH$ and $Q=Y_2VSH$, and  the columns of $Y_1,Y_2$ are orthonormal bases for ${\cal S}$ and ${\cal S}^\perp$ respectively.
Conversely, $[Y_1 \ Y_2]'M= [Y_1'M;Y_2'M]$ is an ordinary GSVD.

By doing this we have a decomposition of $M=P+Q$ such that $P'Q=Q'P=0_{n \times n}$.  Geometrically, instead of decomposing into a ``top half'' and ``bottom half,'' into
a ``horizontal'' and ``vertical'' multiaxis subspace, we are rather allowing for general multiaxes subspaces.
One might think of this as a rotated view of Figure \ref{fig:elip}.  More specifically, most of this paper would take $Y_1=[I;0]$ and $Y_2=[0;I]$, but all that is required is that $Y_1$ and $Y_2$ 
are orthogonal complements.

This geometrical insight underlies an additive decomposition signal processing application found in
\cite{hundley2001solution,hundley2002blind} where $P$ and $Q$ play the role of signal + noise.

\begin{figure}
\hspace*{.8in}
\includegraphics[width=.6\textwidth, angle=30]{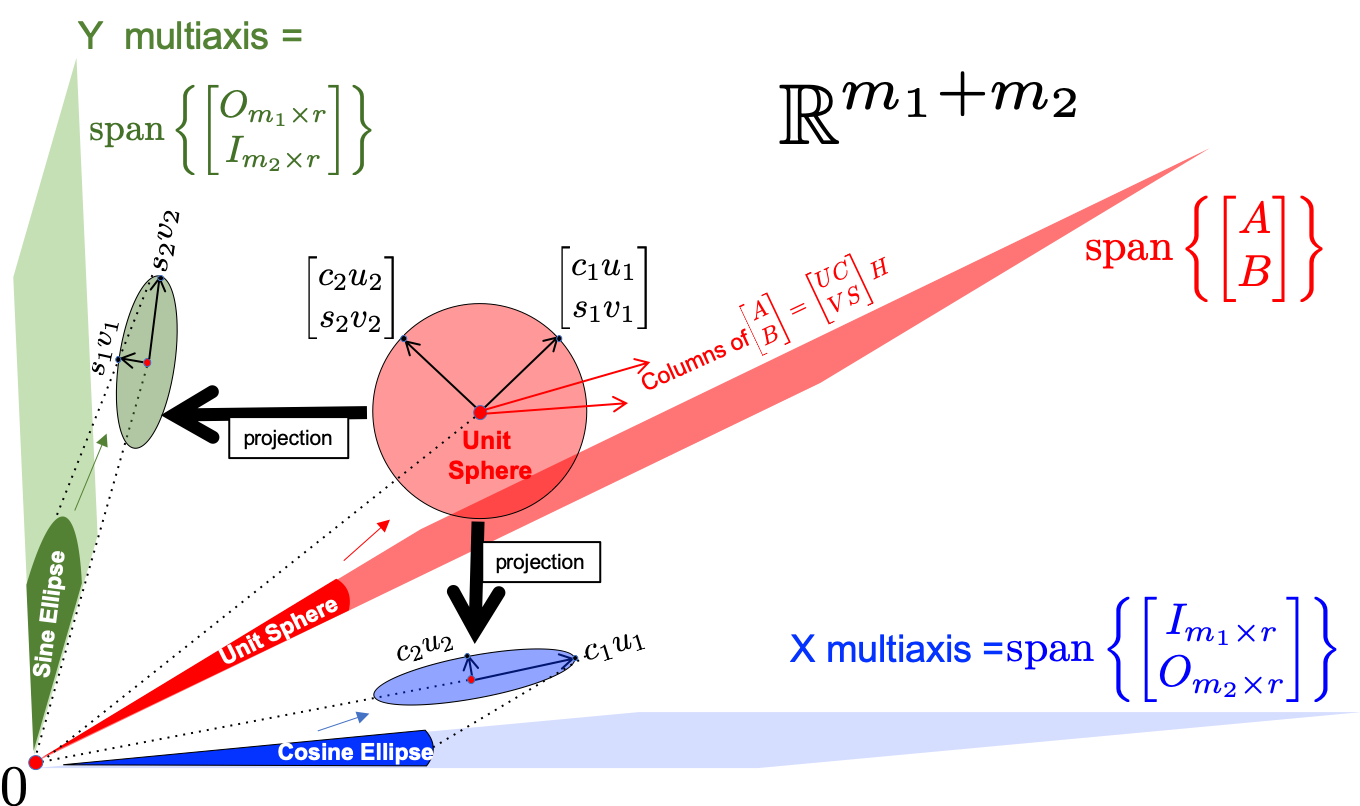}
\caption{\label{fig:rotate}The ellipse picture in Figure \ref{fig:elip} need not fundamentally line up with horizontal and vertical multiaxes. This rotated geometry underlies a signal processing application in \cite{hundley2001solution,hundley2002blind}.}
\end{figure}

\subsection{\texorpdfstring{Orthonormal Bases for $\{Ax : Bx=0\}$ and Friends}{}}

The $U$ matrix of the GSVD provides, in its columns, orthonormal bases for three mutually orthogonal subspaces that arise in many applications:
$$U = \left[ \begin{array}{c|c|c} 
\rule{0in}{.2in} && \\
U_1= & U_2= & U_3= \\
\mbox{orthonormal} & \mbox{completion  to}& \mbox{orthonormal} \\
\mbox{basis for} & \mbox{all of  }   & \mbox{basis for } \\
\{Ax : Bx=0\} & col(A) = \{Ax\} &   col(A)^\perp  \\
\rule{0in}{.2in} && \\
\end{array}\right].
$$
The ``completion'' referred to in the above equation means that taken together, the columns of $U_1$ and $U_2$ form and orthonormal basis for
col($A$). 
From the perspective of Figure \ref{fig:elip}, there are the horizontal directions in the red unit sphere, the generic directions, and the directions that are not present.

\subsubsection{Clustering Matrices}

An important example where the GSVD lurks implicitly or explicitly is clustering.  We will consider an $A$ matrix that indicates the clustering, and a $B$ matrix that indicates
equality of data between the clusters.

We consider data in $\R^p$
and assume a partitioning of $p = p_1+\ldots+p_k$,
into clusters.
The indicator matrix corresponding to the partition of $p$ is :
$$
A = \mbox{Indicator($p_1$,$p_2$,\ldots,$p_k$)} \ =
\begin{array}{c}
\\[-.2in]
 \raisebox{0in}{ $\begin{array}{c} | \\ p_1 \\ | \end{array}$}  \\[.25in]
 \raisebox{0in}{ $\begin{array}{c} | \\ p_2 \\ | \end{array}$}  \\[.2in]
\  \vdots \\[.05in]
   \raisebox{0in}{ $\begin{array}{c} | \\ p_k \\ | \end{array}$} \\[.2in]
  \end{array}
\! \! \! \!
\overbrace{
\left[
\begin{array}{ccccc}
1 &  \cr
 \vdots \cr
1 \cr
& 1 \cr
 & \vdots \cr
& 1 \cr
 && \vdots \cr
&&& 1 \cr
 &&& \vdots \cr
&&& 1 \cr
\end{array}\right]}^k
,
$$
 which we can normalize by setting
$$  Y_1 =  \mbox{Indicator($p_1$,$p_2$,\ldots,$p_k$)} \times  \mbox{Diagonal}( \frac{1}{\sqrt{p_1}},        \frac{1}{\sqrt{p_2}}       ,  \hdots ,  \frac{1}{\sqrt{p_k}}  ).$$
In the Julia computing language, the indicator matrix can be generated succinctly with \verb+A = cat(ones.(Int,partition)...,dims=1:2)+, where \verb+partition+ denotes the vector
$[p_1,\ldots,p_k]$.

The other useful matrix in this context is the constraint matrix whose nullspace is the all ones vector:

$$
B =
\begin{array}{c}
\\[-.1in]
 \raisebox{0in}{ $\begin{array}{c} {\rule{.005in}{.25in}} \\ k-1 \\ \mbox{rows} \\  {\rule{.005in}{.25in}} \end{array}$}  \\[.2in]
  \end{array} 
  \! \! \! \! 
  \overbrace{
\left[ \begin{array}{cccccccc}
1  & 0 & 0 &  \ldots & 0 & 0 & -1 \\
0 & 1  & 0 & \ldots & 0 & 0& -1  \\
\vdots & \vdots & \vdots & \ddots & \vdots & \vdots  & \vdots \\
0 & 0 & 0 & \ldots & 1 & 0 & -1  \\
0 & 0 & 0 & \ldots & 0 & 1 & -1  \\
 \end{array} \right]}^{ \mbox{ $k$ columns}} .
$$

In Julia, with the LinearAlgebra package, this may be written succinctly as  \newline \verb+B = [I -ones(k-1)]+.

Given an $m \times p$ data matrix $D$ there are a number of ``scatter matrices'' that arise that allow us to
compare between clusters and within clusters. Following roughly the notation in   \cite{howland2003structure},
we can partition the data
$$D = [D_1  \ldots D_k], \mbox{ where } D_i \in \R^{m,p_i} \mbox{ and } \sum p_i =p.$$
Let $d_j$ be the $j$th column of $D$ and let $N_i$ denote the column indices in column $i$, $c_i$ is the mean of the columns in cluster $i$, and $c$ is the mean of all the
columns.  The within, between, and mixed scatter matrices are defined as
\begin{align*}
    S_w &= \sum_{i=1}^k  \sum_{j=1}^{N_i}  (d_j - c_i)(d_j-c_i)' \\
    S_b &= \sum_{i=1}^k p_i   (c_i-c)(c_i-c)' \\ 
    S_m &= \sum_{j=1}^n (d_j- c)(d_j - c)' .
\end{align*}
These scatter matrices are readily calculated through the $U$ matrix for the GSVD, one can then set \verb+U, = SVD(A,B)+, where the comma indicates that
we are requesting only the $U$ matrix. We then have that,
\[\ \ \ \ \  \ \ \ \  \  \ \ \ \ \ \ \overbrace{\rule{.4in}{0in}}^{ \mbox{ $1$ col}} 
 \overbrace{\rule{1.3in}{0in}}^{ \mbox{ $k-1$ columns}}
  \overbrace{\rule{1.1in}{0in}}^{ \mbox{ $p-k$ columns} } $$ 
$$U = 
\begin{array}{c}
\\[-.1in]
 \raisebox{0in}{ $\begin{array}{c} {\rule{.005in}{.25in}} \\ p \\ \mbox{rows} \\  {\rule{.005in}{.25in}} \end{array}$}  \\[.2in]
  \end{array} 
  \! \! \! \! 
 \left[ \begin{array}{c|c|c} 
\rule{0in}{.2in} U_1=&& \\
1/\sqrt{p} & U_2= & U_3= \\
1/\sqrt{p} & \mbox{completion  to}& \mbox{orth basis for } \\
 \vdots & \mbox{``between'' clusters  }   & \mbox{``within'' clusters} \\
1/\sqrt{p} & &     \\
\rule{0in}{.2in} && \\
\end{array}\right].
\]
``Completion'' means that $U_1$ and $U_2$ form an orthonormal basis for $A$.
The third block is an orthonormal basis for $A^\perp$.
The ``between'' and ``within'' terms are statistics jargon.  Given a data vector, the first column extracts the normalized mean.  
The next block gives a basis for clustered vectors that are mean-free which by removing the fine details within cluster provides
a way to compare between clusters.  The last block provides the within cluster details.  The number of columns is the dimension
of the space, and in statistics jargon is known as the ``degrees of freedom.''  (See ~\cite[Chap.\ 10]{muirhead}.)

The scatter matrices can be calculated in terms of $U$ using these formulas
\begin{align*}
    S_w &= D (U_3 U_3') D' \\
    S_b &=  D (U_2 U_2') D' \\
    S_m &=  D (I - U_1U_1') D' . 
\end{align*}
One recognizes that the matrices in parentheses in the three expressions above are projection matrices and the orthogonality of $U$ guarantees that $S_w + S_b = S_m$.

\subsubsection{One Way ANOVA made simple}
\label{oneway}
A commonly used statistics test is to decide whether a  proposed clustering of a vector $v$ is justified.  The test takes the average (meaning divide by $k-1$) square component in the $U_2$
direction and divides it by the average (meaning divide by $p-k$) square component in the $U_3$ direction. The following Julia code shows how compactly one can reproduce an  example
from Wikipedia where one can quickly obtain the number computed in Step 5 of  \url{https://en.wikipedia.org/wiki/One-way_analysis_of_variance#Example}.

\vspace{.2in}

{
\hspace{.1in}
\begin{jinput}
\small
\begin{verbatim}
using LinearAlgebra
v = [6,8,4,5,3,4,8,12,9,11,6,8,13,9,11,8,7,12]    # data vector 
A = cat(ones.([6,6,6])...,dims=1:2)               # Indicator(6,6,6)
B = [1 0 -1; 0 1 -1]                              # Constraint matrix
U,= SVD(A,B)                                      # GSVD 
(norm(U[:,2:3]'v)/norm(U[:,4:18]'v))^2 * 15/2     # The F value
\end{verbatim}
\end{jinput}
\small
\begin{verbatim}
     9.264705882352956
\end{verbatim}
}

\vspace{.2in}

While for this problem the classic approach is fine as an algorithm, for general tests for  being in the column space of $A$ but orthogonal to $\{Ax:Bx=0\}$, the GSVD is worth
considering algorithmically and  how we are projecting into the non-horizontal directions is worth understanding geometrically.

\subsubsection{See a slope? Generalize to a GSVD}
\label{seeslope}

In the last line of the above code snippet, the innocent looking 
\begin{verbatim}
                norm(U[:,2:3]'v)/norm(U[:,4:18]'v)
\end{verbatim}
for an orthogonal matrix $U$
carries a message of generalization if you know how to read it. It is a ratio of components in two
orthogonal directions.   You can call it a slope, or a cotangent, or a tangent.  What we called horizontal and vertical multiaxes
in Figure \ref{fig:elip} may now be labeled in this coordinate system: the between and within axes, following the aforementioned statistics nomenclature.

The generalization of the vector $v\in \R^p$ example of Section \ref{fig:elip} is  a $p \times n$ matrix $M$ of data, each data item being one row of length $n$.
It is therefore natural geometrically to consider and interpret the GSVD as
$$
\begin{bmatrix}
U_2'M\\
U_3'M
\end{bmatrix} = \begin{bmatrix}
U_bC\\
U_wS
\end{bmatrix}\cdot H.
$$
The result is $n$ canonical directions for considering between vs within as naturally as comparing human vs yeast, or signal vs noise as we have
seen in previous applications. The multislope, i.e. the generalized singular values (or perhaps we can call this the ANOVA structure)  is $0$ in all but at most $k-1$ directions, owing to the number of columns in $U_2$.

\subsubsection{Discriminant Analysis  Dimension Reduction}
Continuing with the idea in Section \ref{seeslope}.
we observe that it is natural to reduce out  all but the  $k-1$ nonzero ANOVA directions by multiplying
$M$ on the right  by $G = H^\dagger  I_{r,k-1}$ or (for that matter any matrix whose columns span the same subspace of $\R^n$.).  

The reduction to $k-1$ columns
$$[U_2'M;U_3'M] \approx_{\mbox{reduction}}   [U_2'M;U_3'M]G,$$
can be rotated back to the standard coordinate system without any change to the nonzero generalized singular values  (the ANOVA structure) to yield
$$  [U_2\ U_3]   [U_2'M;U_3'M]G =  \left( U_2U_2'M + U_3 U_3'M \right) G  = (I-U_1 U_1') MG,$$
since $UU'=I$. We can reduce the mean also by adding back  $U_1U_1'G$ producing our final reduction, $MG.$

Our simple summary is that for a data matrix $M$, ANOVA measures the nonzero generalized singular values in $[U_2';U_3']M$, a rotated multiaxis system which gives the ratios of the ``between" to the ``within",  and these are the same as for the reduced data matrix $MG$
because we are suppressing the directions with $0$ generalized singular values.

This is a geometrical derivation of an idea and algorithm presented by Park and others \cite{howland2003structure} with a minimization approach.
In their algorithm $G$ can be derived efficiently as the first $k-1$ columns of the $Q$ from the GSVD, and the authors point out that the GSVD
idea is robust even in the case of too little data.

\subsection{The Jacobi Ensemble from Random Matrix Theory is a GSVD} 
Classical random matrix theory centers are Hermite, Laguerre, and Jacobi ensembles.  Historically, they are presented in eigenvalue format, but we have argued that the eigenvalue, SVD, GSVD formats, respectively, are mathematically more natural providing simpler derivations and
clearer insights.
Suppose we have two Gaussian random matrices $A$ ($m_1\times n$) and $B$ ($m_2 \times n$) with $m_1 \geqslant n$ and $m_2 \geqslant n$. For example, \texttt{A=randn(m1,n)} and \texttt{B=randn(m2,n)} using Julia notation. The so-called \emph{MANOVA matrix} (Multivariate Analysis of Variance) is defined
to be 
\begin{equation}
\label{eqn:manova}
(A'A+B'B)^{-1}A'A
\end{equation}
 or in the symmetric form $(A'A+B'B)^{-1/2}A'A(A'A+B'B)^{-1/2}.$
The eigenvalues are the squares of the cosines ($c_i^2$) and are jointly distributed as~\cite{muirhead}
\begin{equation}
c\cdot\prod_{i<j}|\lambda_{i}-\lambda_{j}|^\beta\prod_{i=1}^n\lambda_{i}^{a_1 - p}(1-\lambda_{i})^{a_2 - p},
\end{equation}
where $a_1 = \frac{\beta}{2}m_1, a_2 = \frac{\beta}{2}m_2$ and $p=1 + \frac{\beta}{2}(n-1)$, 
\[
c = \prod_{j=1}^n\frac{\Gamma(1+\frac{\beta}{2})\Gamma(a_1+a_2-\frac{\beta}{2}(n-j))}{\Gamma(1+\frac{\beta}{2}j)\Gamma(a_1-\frac{\beta}{2}(n-j))\Gamma(a_2-\frac{\beta}{2}(n-j))},
\]
where $\beta=1$ for real matrices, $\beta=2$ for complex matrices, $\beta=4$ for quaternion matrices, and general $\beta$ is worth considering, as in~\cite{edelman2010random} . The eigenvalue distribution is known as the \emph{Jacobi ensemble}, which was first referred by name in~\cite{leff1964class}. We refer interested readers to~\cite{edelman2018random}, where the geometrical picture (a simplified version of the ellipse in Figure~\ref{fig:elip}) motivates a direct derivation of the joint density of the Jacobi ensemble. Note that, the direct derivation in~\cite{edelman2018random} fills in a gap stated in Remark 2.3
of~\cite{grinberg2004radon}, where an indirect proof using the Fourier Transform is presented, but a direct proof without the Fourier Transform is desired. An earlier alternative direct proof is due to~\cite{zhang2007radon}.

\section{Mathematical Software}
Suppose one looks up the GSVD in the help pages of your favorite technical computing
language, shown in Table~\ref{tab:GSVD} and the Julia version in Table~\ref{tab:GSVD_julia}. One gets lost in a sea of matrices whose meaning is very hard to
fully appreciate. Surprisingly, we find no standard function for the GSVD in Python (NumPy and SciPy)
though there is some discussion on StackOverflow~\cite{stackoverflow} and Github Numpy issue \#3475\footnote{\url{https://github.com/numpy/numpy/issues/3475}} and scipy issue \#743\footnote{\url{https://github.com/scipy/scipy/issues/743}} and \#1491\footnote{\url{https://github.com/scipy/scipy/issues/1491}}. 

\section{Acknowledgments}

We thank  Orly Alter, Zhaojun Bai, Michael Kirby, Andreas Noack, Chris Paige, Haesun Park, Sri Priya Ponnapalli, Charlie van Loan,  Sabine van Huffel, and Joos Vandewalle for interesting conversations about the GSVD theory, software, and feedback from lectures at the 2017 Householder Symposium, and 2018 SIAM Applied Linear Algebra meeting. We thank Sungwoo Jeong for finding two references in the literature that come close
to the Grassmann viewpoint for the GSVD~\cite{herz1955bessel,grinberg2004radon}. 

We also wish to acknowledge and remember the late Gene Golub, over a decade since his passing, who so effectively promoted the singular value
decomposition.  We remember a time, not so long ago, when the SVD was unheard of outside of  numerical linear algebra circles, and eigenvalues 
were all that were known.  Then like dominos falling, one field after another, biology, economics, fields of engineering, statistics,  computer science, and yes pure mathematics learned
about the value of the SVD as a tool, as an algorithm, and even as vocabulary for effective communication.  Gene with his \fbox{PROF SVD} (Figure~\ref{fig:profsvd}) and \fbox{DR SVD} California vanity license plates seemed always nearby when a field was starting to catch on.  Today the GSVD is as obscure as the SVD was in the early days.  We feel that the GSVD's time has come.  We would be
very pleased if one by one other fields would catch on. 

\begin{figure}
\centering
\includegraphics[width=.3\textwidth]{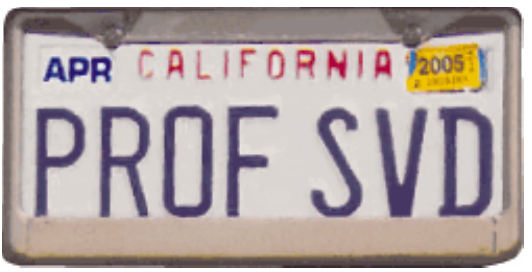}
\caption{The SVD was once an obscure theoretical tool, and now it is everywhere, in part due to the work of the late Gene Golub at Stanford University (Gene Golub's famous vanity license plate illustrated).  It is time for the GSVD to undergo the same
transformation.}
\label{fig:profsvd}
\end{figure}

\lstset{
basicstyle=\small\ttfamily,
columns=flexible,
breaklines=true
}

\begin{table}[ht]
\begin{tabular}{ m{2.5cm}| m{9cm}}
\toprule
 \textsc{language} & \textsc{GSVD documentation in the corresponding language}\\
\midrule
& \\
\textsc{matlab}
(R2018b) &
\url{https://www.mathworks.com/help/matlab/ref/gsvd.html} 
\begin{lstlisting}
[U,V,X,C,S] = gsvd(A,B) returns unitary matrices U and V, a (usually) square matrix X, and nonnegative diagonal matrices C and S so that

    A = U*C*X'
    B = V*S*X'
    C'*C + S'*S = I 

A and B must have the same number of columns, but may have different numbers of rows. If A is m-by-p and B is n-by-p, then U is m-by-m, V is n-by-n, X is p-by-q, C is m-by-q and S is n-by-q, where q = min(m+n,p).
The nonzero elements of S are always on its main diagonal. The nonzero elements of C are on the diagonal diag(C,max(0,q-m)). If m >= q, this is the main diagonal of C.
\end{lstlisting}\\
\midrule
\textsc{Mathematica} (11.3.0) &

\url{https://reference.wolfram.com/language/ref/SingularValueDecomposition.html} \newline {\color{brown}{\verb+>Details and Options+}}.

 \texttt{SingularValueDecomposition[{m,a}]} gives a list of matrices \texttt{\{\{u,ua\},\{w,wa\},v\}} such that \texttt{m} can be written as \texttt{u.w.Conjugate[Transpose[v]]} and \texttt{a} can be written as \texttt{ua.wa.Conjugate[Transpose[v]]}. \\
& \\ 

\midrule

\textsc{R} (geigen v2.2) & 

\url{https://www.rdocumentation.org/packages/geigen/versions/2.2/topics/GSVD}

The matrix $A$ is a $m$-by-$n$ matrix and the matrix $B$ is a $p$-by-$n$ matrix. This function decomposes both matrices; if either one is complex than the other matrix is coerced to be complex. The Generalized Singular Value Decomposition of numeric matrices $A$ and $B$ is given as

$$A = UD_1 [0\ R]Q',\quad\text{and}\quad B = VD_2[0\ R]Q',$$ where 
$U$ an $m\times m$ orthogonal matrix

$V$ an $p\times p$ orthogonal matrix

$Q$ an $n\times n$ orthogonal matrix

$R$ an $r$-by-$r$ upper triangular non singular matrix and the matrix $[0\ R]$ is an $r$-by-$n$ matrix. 

$D_1, D_2$ are quasi diagonal matrices and nonnegative and satisfy $D_1'D_1 + D_2'D_2 = I.$ $D_1$ is an $m$-by-$r$ matrix and $D_2$ is a $p$-by-$r$ matrix.

For details on this decomposition and the structure of the matrices $D_1$ and $D_2$. see \url{http://www.netlib.org/lapack/lug/node36.html}.\\





\bottomrule
\end{tabular}
\label{tab:GSVD}
\caption{The GSVD as portrayed in the documentation of most  technical computing languages seems unlikely to inspire the user unfamiliar with the GSVD.}
\end{table}

\begin{table}[ht]
\begin{center}
\begin{tabular}{ m{2.5cm}| m{9cm}}
\toprule
 \textsc{language} & \textsc{GSVD documentation in corresponding  language}\\
\midrule
 Julia 1.4 (and above) &  
\begin{lstlisting}
 svd(A, B) -> GeneralizedSVD
  Compute the generalized SVD of A and B, returning a GeneralizedSVD factorization object F such that [A;B] = [F.U * F.D1; F.V * F.D2] * F.R0 * F.Q'
    * U is a M-by-M orthogonal matrix,
    * V is a P-by-P orthogonal matrix,
    * Q is a N-by-N orthogonal matrix,
    * D1 is a M-by-(K+L) diagonal matrix with 1s in the first K entries,
    * D2 is a P-by-(K+L) matrix whose top right L-by-L block is diagonal,
    * R0 is a (K+L)-by-N matrix whose rightmost (K+L)-by-(K+L) block is nonsingular upper block triangular,
  K+L is the effective numerical rank of the matrix [A; B].
  Iterating the decomposition produces the components U, V, Q, D1, D2, and R0.
  The generalized SVD is used in applications such as when one wants to compare how much belongs to A vs. how much belongs to B, as in human vs yeast genome, or signal vs noise, or between clusters vs within clusters. (See Edelman and Wang for discussion: https://arxiv.org/abs/1901.00485)
  It decomposes [A; B] into [UC; VS]H, where [UC; VS] is a natural orthogonal basis for the column space of [A; B], and H = RQ' is a natural non-orthogonal basis for the rowspace of [A;B], where the top rows are most closely attributed to the A matrix, and the bottom to the B matrix. The multi-cosine/sine matrices C and S provide a multi-measure of how much A vs how much B, and U and V provide directions in which these are measured.
\end{lstlisting}
 \\
\bottomrule
\end{tabular}
\caption{Documentation in Julia 1.4 (and above) with the original pull request \url{https://github.com/JuliaLang/julia/pull/30239}.}
\label{tab:GSVD_julia}
\end{center}
\end{table}

\bibliographystyle{siamplain}
\bibliography{jacobibib}

\newpage
\begin{appendices}

\section{An In-depth Discussion of $U, V, C, S, H$}
\label{sec:appendix}
\subsection{\texorpdfstring{The square orthogonal matrices $U$ and $V$}{}}

The $U$ and $V$ matrices represent orthogonal bases for
$\R^{m_1}$ and $\R^{m_2}$ respectively. 

 One obtains an orthogonal basis for the column space of $A$ ($B$) by taking the columns of $U$ ($V$) corresponding to the non-zero rows of $C$ ($S$).  The remaining columns are an orthogonal basis for the left nullspace. (Recall, in the ordinary SVD, the ``$U$''  can be chosen to be square, and one  can divide ``$U$'' into the column space/left nullspace through  ``$\Sigma$'' in the analagous way.)
 
 We see in Section~\ref{sec:ellipse} that the columns of $U$ and the columns of $V$ may be thought of as semi-axes of ellipses (with the possibility of degenerate axes.)


\subsubsection{\texorpdfstring{The diagonal cosine and one-diagonal sine matrices: $C$ and $S$}{}}
\label{cands}

The cosines $1 \ge c_1 \ge \ldots \ge c_r \ge 0$ and
sines $0 \le s_1 \le \ldots \le s_r \le 1$ satisfy  $c_i^2+s_i^2=1$.
They represent the lengths of the semi-axes of our two ellipses.
The generalized singular values are the cotangents $\sigma_i = c_i/s_i$ which may be $0$ or infinite.  When $0 < \sigma_i < \infty$, we say that $\sigma_i$ is finite. 

As show in Figure~\ref{fig:CS}, the cosine matrix $C\in\R^{m_1, r}$ 
matrix puts the $c_i$ on the diagonal starting with $c_1$ in the (1,1) position. If we run out of room, by not having enough rows, we drop some of the $0$ cosines.

The sine matrix $S\in\R^{m_2, r}$ puts the $s_i$ on {\em some}
diagonal, and again if we run out of room, by not having enough rows, we drop some of the $0$ sines.  One convention (used by LAPACK \cite{anderson1999lapack}) 
puts all the positive $s_i$ in the top rows
by putting the positive diagonal in the top right corner. Another \cite[Eq. 2.3]{paige1981towards}  puts them in the bottom rows which as Paige and Saunders remark (and we agree) creates
\cite[p.401 top]{paige1981towards}:  an  ``easy [way] to remember [the] symmetry."

Either way $C'C$ and $S'S$ are square $r \times r$ diagonal with the $c_i^2$ and $s_i^2$ on the main diagonal and $C'C+S'S=I_r.$
The only difference between the two conventions is where the 
orthogonal basis for the column space of $B$ ends up in the columns of $V$, i.e., the left or right side.
(It is always those columns of $V$ that correspond to the rows where an $s_i>0$.)
When the significant elements of $S$ are on top, the column space basis is on the left like it is with $U$.  When it is on the bottom, one feels that
the $B$ is being treated as something of a ``mirror image" of $A$ with the column space basis on the right of $V$, and $S$ being something of a 180 degree rotation
(in structure) from $C$. 

\subsection{\texorpdfstring{The matrix $H$ that has no orthogonality or diagonal properties}{}}

On a first glance, no self-respecting  decomposition in the
SVD family should be neither diagonal nor orthogonal.
Nonetheless, all we can say about $H$ is that 
 $H\in\R^{r\times n}$ is a full row rank matrix whose rowspace is that of $[A;B]$.
 A very common case is $r=n$ in which case $H$ is square non-singular.
 
 In the same way that a vector is specified by its direction and length, 
 we think of $[A;B]$ as being specified by its column space and the rest of the information.  The matrix $H$ specifies the rest of the information.

Rowspace information is available in $H$.  The first $r_a$ rows of $H$ form a basis for the rowspace of $A$.  The last $r_b$ rows of $H$ form a basis for the rowspace of $B$.  The nullspaces are not as immediately available due to the non-orthogonality of $H$.  The nullspace of $H$ is the common nullspace of $A$ and $B$.
Of course one can use a QR decomposition.

\subsection{Compact Formats}
\label{compactformat}

One can optionally delete all the zero rows of $C$ or $S$ and the corresponding columns of $U$ and $V$.  This kills the left nullspace basis vectors, but preserves the column space vectors.

\subsection{Expanded Format}

When $r<n$ one can add $n-r$ zero columns to both $C$ and $S$ and expand $H$ to a full square non-singular matrix by adding any $n-r$ rows to $H$ that would make it invertible.
 
\subsection{Further reduction to Orthogonal and Triangular}
\label{orthtri}

The Expanded $H$ matrix can be written $[0 \  R]Q'$, where $R$ is triangular $\R^{r,r}$, and $Q$ is square orthogonal $\R^{n,n}$. In this case the initial $n-r$ columns of $Q$ are an orthogonal basis for the common nullspace of $A$ and $B$. 

\end{appendices}

\end{document}